\documentclass{amsart}
\usepackage[a4paper]{geometry}
\usepackage{amsthm}
\usepackage{url}
\usepackage{comment}
\makeatletter

\def\part{\@startsection{part}{0}%
 \z@{\linespacing\@plus\linespacing}{.5\linespacing}%
{\normalfont\bfseries\centering}}
\makeatother

\newcommand{\R}{\mathbb{R}}

\newtheorem{theorem}{Theorem}[section]
\newtheorem{lemma}[theorem]{Lemma}
\newtheorem{corollary}[theorem]{Corollary}
\newtheorem{proposition}[theorem]{Proposition}

\theoremstyle{remark}
\newtheorem{remark}[theorem]{Remark}

\theoremstyle{definition}

\usepackage{hyperref}
\usepackage{xcolor}
\hypersetup{
 colorlinks = true,
urlcolor= blue,
 linkcolor = blue,
 citecolor= red
}

\usepackage{setspace}
\usepackage{amssymb}

\title{Sharp Estimates for some Integral-Geometric Quantities RELATED TO TRANSVERSALITY, CURVATURE AND VISIBILITY}

\author{Silouanos Brazitikos}

\author{Dimitris-Marios Liakopoulos}

\address{Department of Mathematics \& Applied Mathematics, University of Crete, Voutes Campus, 70013 Heraklion, Greece}

\email{silouanb@uoc.gr, dimliako1@gmail.com }

\thanks{}

\date{\today}

\begin{document}
\begin{abstract}
 We investigate integral-geometric quantities arising from harmonic analysis which measure visibility and transversality. Motivated by their applications in multilinear Kakeya problems and affine-invariant measures on surfaces, we derive exact lower and upper bounds employing geometric and functional inequalities of convex geometry.
\end{abstract}
\maketitle

\section{Introduction}

In the intersection of convex geometry, integral geometry, and harmonic analysis, certain fundamental quantities have recently emerged that capture essential geometric properties of surfaces and vector fields. A central object of study in this paper, as well as in recent literature, is a family of integral-geometric quantities that measure the global transversality of vector fields.

Let $\mathbb{S} = (S, \sigma, v)$ be a generalised $d$-hypersurface, defined as a triple where $(S, \sigma)$ is a $\sigma$-finite measure space and $v: S \rightarrow \mathbb{R}^d$ is a measurable vector field. For a $j$-tuple of such hypersurfaces, $(\mathbb{S}_1, \dots, \mathbb{S}_j)$ with $j \le d$, the quantities $Q_j^p$ for $p > 0$ were defined in \cite{BrazitikosCarberyMacIntyre2023} as:
$$
Q_{j}^{p}(\mathbb{S}_{1},\dots,\mathbb{S}_{j}) := \left( \int_{S_{j}}\cdot\cdot\cdot\int_{S_{1}}|v_{1}(x_{1})\wedge\cdot\cdot\cdot\wedge v_{j}(x_{j})|^{p}d\sigma_{1}(x_{1})...d\sigma_{j}(x_{j}) \right)^{1/jp}
$$
where $|v_1 \wedge \dots \wedge v_j|$ denotes the $j$-dimensional volume of the parallelotope spanned by the vectors $v_1, \dots, v_j$. These quantities effectively provide an $L^p$ average of the joint transversality of the vector fields over the product space $S_1 \times \dots \times S_j$. In the diagonal case, where $\mathbb{S}_1 = \dots = \mathbb{S}_j = \mathbb{S}$, we use the abbreviated notation $Q_j^p(\mathbb{S})$. These quantities have found significant applications, particularly in the context of multilinear Kakeya problems and in the development of affine-invariant measures on surfaces.

A foundational result in the study of these quantities was recently provided in \cite{BrazitikosCarberyMacIntyre2023}. The authors established a general Finner-type inequality which provides an upper bound for $Q_j^p$ in terms of quantities with a smaller number of vector fields. Specifically, if $A_i \subseteq \{1, \dots, j\}$ and $\alpha_i$ are positive numbers  for $1 \leq i \leq m$, such that for all $1 \leq l \leq j$, $\sum_{i=1}^m \alpha_i \chi_{A_i}(l) = 1$, then 
$\{(A_i, \alpha_i)\}_{i=1}^m$ is called a {\em uniform cover} of $\{1, \dots , j\}$. Moreover, if $A \subseteq \{1, \dots , j\}$, let $\Pi_A (\mathbb{S}_1, \dots , \mathbb{S}_j) = (\mathbb{S}_n)_{n \in A}$ be a projection. The authors of \cite{BrazitikosCarberyMacIntyre2023} proved that
\begin{equation}\label{affine_finner}
Q_{j}^{p}(\mathbb{S}_{1},\dots,\mathbb{S}_{j})\le\prod_{i=1}^{m}Q_{k_{i}}^{p}(\Pi_{A_{i}}(\mathbb{S}_{1},\dots,\mathbb{S}_{j}))^{\alpha_{i}k_{i}/j}.
\end{equation}
The proof of this inequality relies on a geometric lemma concerning the volume of parallelotopes, followed by an application of the abstract Finner inequality on the resulting product of functions. While this inequality is sharp in general, its sharpness is typically achieved for systems of mutually orthogonal vectors.

Our primary contribution is to provide a refinement of this inequality, which is sensitive to the angular configuration of the vectors involved. This result, Theorem \ref{thm:rho-sup}, establishes a tighter upper bound by introducing a term that explicitly quantifies the geometric arrangement of the vector fields. We show that:
\begin{equation}\label{angle_sensitive}
Q_{j}^{p}(\mathbb{S}_{1},\dots,\mathbb{S}_{j})\le\left(\prod_{i=1}^{m}Q_{k_{i}}^{p}(\Pi_{A_{i}}(\mathbb{S}_{1},\dots,\mathbb{S}_{j}))^{\alpha_{i}k_{i}}\right)^{1/(jp)}\cdot \sup_{x\in\Pi_{n}S_{n}}\rho(x)^{1/j}.
\end{equation}
The novel factor $\rho(x)$ is defined through the Gram matrix of the unit directions of the vectors:
$$
\rho(x):=\frac{(\det C(x))^{1/2}}{\prod_{i=1}^{m}(\det C_{A,i}(x))^{\alpha,/2}},$$
where $C(x)$ is the Gram matrix of the unit vectors $\{\hat{v}_n(x_n)\}_{n=1}^j$ and $C_{A_i}(x)$ are its principal submatrices. This factor $\rho(x)$ serves as an ``angular deficit" term; it is equal to 1 when the vectors are orthogonal and strictly less than 1 otherwise. 

Our proof operates on a local level first. We begin with the classical Gram identity, which allows us to establish an exact local identity at each point $x = (x_1, \dots, x_j)$:
$$
|v_{1}\wedge\cdot\cdot\cdot\wedge v_{j}|^{p}=(\prod_{i=1}^{m}|\bigwedge_{n\in A_{i}}v_{n}|^{\alpha_{i}p})\rho(x)^{p}
$$
This identity elegantly separates the norms of the projected sub-parallelotopes from the angular dependence captured by $\rho(x)$. The final inequality is then obtained by integrating this local identity and applying the abstract Finner inequality to the product of functions $F_i = |\bigwedge_{n \in A_i} v_n|^p$. This local-to-global approach is the key to preserving the geometric information contained in the angles between the vectors.

The second contribution of our present work is to prove a generalisation of this inequality, when $p=1$, concerning mixed volumes, starting from an inequality proved in \cite[Theorem~1.5]{GBBC}, improving an inequality in \cite[Theorem~1.2]{BrazitikosGiannopoulosLiakopoulos+2018+345+354}. This appears here as Theorem \ref{mixed_vol_theorem}.  

Another central theme in \cite{BrazitikosCarberyMacIntyre2023} is the study of sharp inequalities in the diagonal case, particularly for $p=1, 2$. A key result there, Theorem 1.2, demonstrates that the sequence $Q_j^p(\mathbb{S})$ exhibits a specific monotonicity property, with the unit sphere $\mathbb{S}^{d-1}$ acting as the extremizer:
$$
Q_{j+1}^{p}(\mathbb{S})\le\frac{Q_{j+1}^{p}(\mathbb{S}^{d-1})}{Q_{j}^{p}(\mathbb{S}^{d-1})}Q_{j}^{p}(\mathbb{S}), \quad \text{for } p \in \{1, 2\}.
$$
This result is derived from the deep log-concavity properties of mixed volumes (for $p=1$) and mixed discriminants (for $p=2$), as captured by the Aleksandrov-Fenchel and Aleksandrov inequalities, respectively.

Our work approaches the problem for other values of $p$. We restrict ourselves to the case where we have a general measure on the sphere and we establish a general sharp inequality, showing that for $1 < p < 2$, the uniform measure is the unique maximizer, whereas for $p > 2$, it is not a maximizer. This generalises the work of Tilli, \cite[Theorem~4.1]{Tilli}, where the planar case was considered, only for $p=1$. Moreover, we show that for $p>2$ the maximizer is a discrete measure on the sphere supported on $\{\pm e_1,\dots,\pm e_d\}$.

In the last section, we provide sharp bounds for visibility. In \cite{BrazitikosCarberyMacIntyre2023}  the quantities $Q^p_j(\mathbb{S})$ were related to the notion of visibility which has arisen in harmonic analysis \cite{Guth} in connection with the multilinear Kakeya problem.

\medskip
\noindent
Let $\mathbb{S}= (S, \sigma, v)$ be a generalised $d$-hypersurface, let $1 \leq p < \infty$ and let 
$$ K^p = K^p(\mathbb{S}) := \left\{ y \in \mathbb{R}^d \, : \, \left(\int_S |\langle y, v(x)\rangle|^p {\rm d} \sigma(x)\right)^{1/p} \leq 1\right\}.$$
Then $K^p$ is a closed, balanced and convex subset of $\mathbb{R}^d$ which has nonempty interior, and which under certain mild conditions on $\mathbb{S}$ will be compact. The {\bf $p$-visibility} of $\mathbb{S}$ is defined by
$${\rm vis}^p(\mathbb{S}) := {\rm vol}(K^p(\mathbb{S}))^{-1/d}.$$
Note that this definition differs from some of the literature (\cite{Guth, Zhang, Z-K}), where, in the case $p=1$, ${\rm vis}^{1}(\mathbb{S})$ is taken to be ${\rm vol}(K^{1}(\mathbb{S}))^{-1}$ rather than ${\rm vol}(K^{1}(\mathbb{S}))^{-1/d}$. The following was proved in~\cite{BrazitikosCarberyMacIntyre2023}.

\begin{proposition}\label{prop:visplanes}
Let $k_1 + \dots + k_m = d$. We have
$$ {\rm vis}(\mathbb{S}) \sim_d \inf_{E_j \in \mathcal{G}_{d, k_j}} \left(\frac{\sigma(E_1, \mathbb{S}) \dots \sigma(E_m, \mathbb{S})}{|E_1 \wedge \dots \wedge E_m|}\right)^{1/d},$$
and the infimum is essentially achieved when each $E_j$ is the span of some $k_j$ vectors from the principal directions of the John ellipsoid of $K(\mathbb{S})$. 
\end{proposition}
Applying this to $m=1$ we obtain the following characterization of the quantity $Q_d^1(\mathbb{S})$ in terms of visibility: 
\begin{corollary}\label{cor:visaswedge}
We have
$$ {\rm vis}(\mathbb{S}) \sim_d \left(\int_S \dots \int_S |v(x_1) \wedge \dots \wedge v(x_d)| {\rm d} \sigma(x_1) \dots  {\rm d} \sigma(x_d)\right)^{1/d} = \; Q_d^1(\mathbb{S}).$$
\end{corollary}

In our work, we establish a sharp version, in terms of the constants that depend on $d$ and $k_i$. One of the ingredients of the proof in \cite{BrazitikosCarberyMacIntyre2023} was an inequality from \cite[Theorem~3.1]{Zhang}. We first observe that this inequality is in fact the reverse Santal\'
{o} or Bourgain Milman inequality, see for example \cite[ Theorem~8.2.2]{GianMilArt}, which provides a lower bound for the product of the volume of a convex body and its polar.  While it is an open problem to find the exact lower bound for general convex bodies, it is known in the class of generalised zonoids, which is the case for $\Pi\mathbb{S}$. Moreover, the fact that the class of extremisers of this inequality are also extremisers for the affine-invariant Loomis--Whitney inequality allows us to obtain a sharp inequality for the upper bound. 

On the other hand, for the lower bound, instead of using the axes of the John ellipsoid, we use the axes of the Binet ellipsoid. Utilizing the isotropicity present there, with the aid of a refinement of the Continuous Brascamp--Lieb inequality, see Lemma \ref{cont_brascamp_lieb}, we are able to obtain a sharp lower bound. We note that the lower bound is also connected with the reverse Loomis--Whitney inequalities from \cite{KoldobskySaroglouZvavitch2019} and \cite{alonso2021reverse}. Here we prove such an inequality for all zonoids and $k$-dimensional  projections, see Theorem \ref{reverse_loomis_zonoids}. 

Finally, we obtain sharp bounds for ${\mathrm {vis}}^p$ for other values of $p$. In Theorem \ref{ell_brascamp_lieb}, a sharp inequality for $p=2$ is obtained, while in Theorem~\ref{thm:vis-vs-I}, we prove that in general $$ {\rm {vis}}^p(\mathbb{S}) \sim_{d,p} \left(\int_S \dots \int_S |v(x_1) \wedge \dots \wedge v(x_d)|^p {\rm d} \sigma(x_1) \dots  {\rm d} \sigma(x_d)\right)^{1/pd} = \; Q_d^p(\mathbb{S}).$$ 
The proof proceeds via Lewis's position. The advantage there is that the transformation preserves volume, so it does not change ${\rm vis^p}$, and, in a sense, places the body in an isotropic position, allowing the use of the continuous Brascamp–Lieb inequality.
 
\section{Proof of Inequality \eqref{angle_sensitive}}\label{section-2}
\noindent In this section, for a fixed $p$, we provide an upper bound for the quantity $Q_j^p(\mathbb{S})$ by the geometric means of the quantities $Q_i^p(\mathbb{S})$ for $i<j$ and a quantity that is sensitive to the angles of the corresponding vectors. This is a refinement of inequality \eqref{affine_finner}.

\begin{theorem}\label{thm:rho-sup}
For each $x=(x_1,\dots,x_j)\in\prod_{n=1}^j S_n$ define the unit vectors
\[
\hat v_n(x_n):=\begin{cases}v_n(x_n)/|v_n(x_n)|,&v_n(x_n)\neq0,\\ \tilde u_n,&v_n(x_n)=0,\end{cases}
\]
where each $\tilde u_n$ is an arbitrary unit vector. Let
\[
C(x):=\big(\langle\hat v_a(x_a),\hat v_b(x_b)\rangle\big)_{a,b=1}^j
\]
be the Gram matrix of the unit directions and for each $A_i$ denote by $C_{A_i}(x)$ the corresponding principal submatrix. Define
\begin{equation}\label{eq:rho-def}
\rho(x):=\dfrac{(\det C(x))^{1/2}}{\displaystyle\prod_{i=1}^m (\det C_{A_i}(x))^{\alpha_i/2}}.
\end{equation}
Then the following inequality holds:
\begin{equation}\label{eq:main-sup}
Q^p_j(\mathbb S_1,\dots,\mathbb S_j) \;\le\; \bigg(\prod_{i=1}^m Q^p_{k_i}(\Pi_{A_i}(\mathbb S_1,\dots,\mathbb S_j))^{\alpha_i k_i}\bigg)^{1/(jp)}\;\cdot\;\sup_{x\in\prod_{n}S_n}\rho(x)^{1/j}.
\end{equation}
In particular,
\[
 Q^p_j(\mathbb S_1,\dots,\mathbb S_j) \le \sup_{x}\rho(x)^{1/j}\;\prod_{i=1}^m Q^p_{k_i}(\Pi_{A_i}(\mathbb S))^{\alpha_i k_i / j},
\]
where
\[
\sup_x \rho(x) \;=\; 
\sup_{(u_1,\dots,u_j)\in \mathcal V_1\times\cdots\times\mathcal V_j}
\dfrac{\bigl(\det(\langle u_a,u_b\rangle)_{a,b}\bigr)^{1/2}}
{\displaystyle\prod_{i=1}^m \bigl(\det(\langle u_a,u_b\rangle)_{a,b\in A_i}\bigr)^{\alpha_i/2}}.\]
\end{theorem}
\begin{proof}
The classical Gram identity yields, at any point $x$ where all $v_n(x_n)\neq0$,
\begin{equation}\label{eq:wedge-gram}
|v_1(x_1)\wedge\cdots\wedge v_j(x_j)|^2 = \Big(\prod_{n=1}^j |v_n(x_n)|^2\Big)\det C(x).
\end{equation}
Consequently, for any $p>0$,
\begin{equation}\label{eq:wedge-p}
|v_1\wedge\cdots\wedge v_j|^p = \Big(\prod_{n=1}^j |v_n|^p\Big)\big(\det C\big)^{p/2}.
\end{equation}
For any subset $A\subset\{1,\dots,j\}$ the same identity applied to the vectors indexed by $A$ gives
\begin{equation}\label{eq:block-wedge}
\Big|\bigwedge_{n\in A} v_n\Big|^p = \Big(\prod_{n\in A}|v_n|^p\Big)\big(\det C_A\big)^{p/2},
\end{equation}
where $C_A$ denotes the corresponding principal submatrix of $C$.

Using \eqref{eq:block-wedge} for each $A_i$ and multiplying the powers $\alpha_i$, and observing that the condition $\sum_i\alpha_i\chi_{A_i}(n)=1$ causes the product of scalar norms to cancel, we arrive at the exact local identity
\begin{equation}\label{eq:local-identity}
|v_1\wedge\cdots\wedge v_j|^p = \Big(\prod_{i=1}^m \Big|\bigwedge_{n\in A_i} v_n\Big|^{\alpha_i p}\Big)\,\rho(x)^p,
\end{equation}
where $\rho(x)$ is defined by \eqref{eq:rho-def}. The verification is direct: multiplying the right-hand side yields the full product of $|v_n|^p$ and the ratio of determinants which is precisely $\rho(x)^p$.

Define for each $i$ the function
\[
F_i\big((x_n)_{n\in A_i}\big):=\Big|\bigwedge_{n\in A_i} v_n(x_n)\Big|^p.
\]
Integrating identity \eqref{eq:local-identity} over $\prod_n S_n$ with respect to the product measure $d\mu:=\prod_{n=1}^j d\sigma_n(x_n)$ yields
\begin{align*}
Q_j(\mathbb S_1,\dots,\mathbb S_j)^{jp} &= \int_{\prod_n S_n} |v_1\wedge\cdots\wedge v_j|^p\,d\mu(x) \\
&= \int_{\prod_n S_n} \Big(\prod_{i=1}^m F_i((x_n)_{n\in A_i})^{\alpha_i}\Big)\,\rho(x)^p\,d\mu(x).
\end{align*}
By the abstract Finner inequality we have
\[
\int \prod_{i=1}^m F_i^{\alpha_i}\,d\mu \le \prod_{i=1}^m \Big(\int_{\prod_{n\in A_i}S_n} F_i\,\prod_{n\in A_i}d\sigma_n\Big)^{\alpha_i}.
\]
Using that $\rho(x)^p\le (\sup_x\rho(x))^p$ for all $x$ we get
\begin{align*}
Q_j(\mathbb S)^{jp} &= \int \Big(\prod_i F_i^{\alpha_i}\Big)\,\rho^p\,d\mu \\
&\le (\sup_x\rho(x))^p\int \prod_i F_i^{\alpha_i}\,d\mu \\
&\le (\sup_x\rho(x))^p\prod_{i=1}^m \Big(\int_{\prod_{n\in A_i}S_n} F_i\,\prod_{n\in A_i}d\sigma_n\Big)^{\alpha_i}.
\end{align*}
Taking the $(jp)$-th root and recalling that $\int_{\prod_{n\in A_i}S_n} F_i=Q_{k_i}(\Pi_{A_i}(\mathbb S))^{k_i p}$ yields inequality \eqref{eq:main-sup}.

\end{proof}

\begin{remark}
The proof gives immediately a stronger formulation where the factor $\sup_x\rho(x)^{1/j}$ may be replaced by the more precise
\[
\mathcal R:=\Big(\int \Big(\prod_{i=1}^m \Big(\frac{F_i}{\int F_i}\Big)^{\alpha_i}\Big)\rho^p\,d\mu\Big)^{1/(jp)}\le \sup_x\rho(x)^{1/j},
\]
so that one has the refined inequality
\[
Q_j(\mathbb S) \le \Big(\prod_{i=1}^m Q_{k_i}(\Pi_{A_i}(\mathbb S))^{\alpha_i k_i}\Big)^{1/(jp)}\cdot \mathcal R.
\]
The factor $\mathcal R$ depends explicitly on the Gram matrices (hence on the mutual directions of the $v_n$) and quantifies the angular deficit compared with the angle-blind bound.
\end{remark}

\section{Generalised B\'{e}zout Type Inequality}\label{section-3}
\noindent In this section, we give a B\'{e}zout-type inequality concerning mixed volumes of an arbitrary convex body $K\subset\mathbb{R}^d$ and a finite family of zonoids. An immediate consequence is an upper bound for the quantity $Q_j^1(\mathbb{S})$.

We introduce some terminology and notation. For every non-empty
$\tau\subset [d]:=\{1,\dots,d\}$ we set $F_{\tau}=\operatorname{span}\{e_j : j\in\tau\}$ and $E_{\tau}=F_{\tau}^{\perp}$, where $e_1,\dots,e_d$ is the standard basis of $\mathbb{R}^d$. Given $s\ge1$ and $\sigma\subset[d]$ we say that (not necessarily distinct) sets $\sigma_1,\dots,\sigma_r\subseteq\sigma$ form an \emph{$s$-uniform cover} of $\sigma$ if every $j\in\sigma$ belongs to exactly $s$ of the sets $\sigma_i$.

We start with the affine local Loomis--Whitney inequality proved in \cite{GBBC}.
\begin{theorem}[Affine local Loomis--Whitney, \cite{GBBC}]\label{thm:affine-llw}
Let $\{w_1,\dots,w_d\}$ be a basis of $\mathbb{R}^d$. Let~$r\ge1$ and let $(\sigma_1,\dots,\sigma_r)$ be a uniform cover of $\sigma\subseteq[d]$ with weights $(p_1,\dots,p_r)$. Set
$$F_{\sigma}=\operatorname{span}\{w_j:j\in\sigma\},\qquad F_{\sigma_i}=\operatorname{span}\{w_j:j\in\sigma_i\},$$
and denote $d_i:=|\sigma_i|$ and $p:=\sum_{i=1}^rp_i$. For $E_{\sigma}=F_{\sigma}^{\perp}$ and $E_{\sigma_i}=F_{\sigma_i}^{\perp}$ the following holds for every convex body $K\subset\mathbb{R}^d$:
\begin{equation}\label{eq:affine-llw}
|K|^{p-1}|P_{E_{\sigma}}(K)|\le\frac{\prod_{i=1}^r|\wedge_{j\in\sigma_i}w_j|^{p_i}}{|\wedge_{j\in\sigma}w_j|}\frac{\prod_{i=1}^{r}\binom{d-d_i}{d-|\sigma|}^{p_i}}{{\binom{d}{|\sigma|}}^{p-1}}\prod_{i=1}^{r}|P_{E_{\sigma_i}}(K)|^{p_i}.
\end{equation}
\end{theorem}
We also need the following useful formula for mixed volumes.
\begin{lemma}\label{lem:product-mixed}
Let $E\in G_{d,k}$ and let $L_1,\dots,L_{d-k}$ be compact convex subsets of $E^{\perp}$. If $K_1,\dots,K_k$ are convex bodies in $\mathbb{R}^d$, then
\begin{equation}\label{eq:mix-prod}
\binom{d}{k}V(K_1,\dots,K_k,L_1,\dots,L_{d-k})=V_E(P_E(K_1),\dots,P_E(K_k))\,V_{E^{\perp}}(L_1,\dots,L_{d-k}).
\end{equation}
\end{lemma}

Using the two results above, we prove the following inequality for mixed volumes.
\begin{theorem}\label{mixed_vol_theorem}
Let $r>s\ge1$, let $\sigma\subset [d]$, and let $(\sigma_1,\dots,\sigma_r)$ be an $s$-uniform cover of $\sigma$. For every convex body $K$ and zonoids $Z_j$ ($j\in\sigma$) in $\mathbb{R}^d$ we have
\begin{equation}\label{eq:main-mixed}
|K|^{r-s}V\bigl(K[d-|\sigma|],(Z_j)_{j\in\sigma}\bigr)^s\le \frac{\prod_{i=1}^{r}\binom{d-d_i}{d-|\sigma|}\binom{d}{d_i}}{\binom{d}{|\sigma|}^r}\prod_{i=1}^{r}V\bigl(K[d-|\sigma_i|],(Z_j)_{j\in\sigma_i}\bigr),
\end{equation}
where $d_i:=|\sigma_i|$ for each $i$.
\end{theorem}

\begin{proof}
Let $\{w_i\}_{i=1}^d$ be a (not necessarily orthonormal) basis of $\mathbb{R}^d$. With the notation of Theorem~\ref{thm:affine-llw}, apply Lemma~\ref{lem:product-mixed} with $k=d-|\sigma|$, $K_1=\cdots=K_{d-|\sigma|}=K$, and $L_j=[0,w_j]$ for $j\in\sigma$. This yields
\begin{equation}\label{eq:proj-sigma}
|P_{E_{\sigma}}(K)|=\frac{1}{|\wedge_{j\in\sigma}w_j|}\binom{d}{|\sigma|}V\bigl(K[d-|\sigma|],([0,w_j])_{j\in\sigma}\bigr).
\end{equation}
Similarly, for each $i$,
\begin{equation}\label{eq:proj-sigma-i}
|P_{E_{\sigma_i}}(K)|=\frac{1}{|\wedge_{j\in\sigma_i}w_j|}\binom{d}{d_i}V\bigl(K[d-|\sigma_i|],([0,w_j])_{j\in\sigma_i}\bigr).
\end{equation}
Applying Theorem~\ref{thm:affine-llw} with $p_i=1/s$ for all $i$ (so that $p=r/s$) gives
\begin{equation}\label{eq:apply-llw}
|K|^{\frac{r}{s}-1}|P_{E_{\sigma}}(K)|\le \frac{\prod_{i=1}^r|\wedge_{j\in\sigma_i}w_j|^{\frac{1}{s}}}{|\wedge_{j\in\sigma}w_j|}\frac{\prod_{i=1}^{r}\binom{d-d_i}{d-|\sigma|}^{\frac{1}{s}}}{{\binom{d}{|\sigma|}}^{\frac{r}{s}-1}}\prod_{i=1}^{r}|P_{E_{\sigma_i}}(K)|^{\frac{1}{s}}.
\end{equation}

Substitute \eqref{eq:proj-sigma} and \eqref{eq:proj-sigma-i} into \eqref{eq:apply-llw}, then raise both sides to the power $s$. After simplification one obtains
\begin{equation}\label{eq:segment-bound}
|K|^{r-s}V\bigl(K[d-|\sigma|],([0,w_j])_{j\in\sigma}\bigr)^s
\le\frac{\prod_{i=1}^{r}\binom{d-d_i}{d-|\sigma|}\binom{d}{d_i}}{\binom{d}{|\sigma|}^r}\prod_{i=1}^{r}V\bigl(K[d-|\sigma_i|],([0,w_j])_{j\in\sigma_i}\bigr).
\end{equation}

Mixed volumes are invariant under translations thus, if $L_j=[x_j,y_j]$ are arbitrary line segments and we set $z_j=y_j-x_j$, then
$$V\bigl(K[d-|\sigma|],(L_j)_{j\in\sigma}\bigr)=V\bigl(K[d-|\sigma|],([0,z_j])_{j\in\sigma}\bigr).$$
Hence \eqref{eq:segment-bound} holds for any family of line segments whose directions are pairwise linearly independent, so we have
\begin{equation}\label{eq:line-bound}
    |K|^{r-s}V\bigl(K[d-|\sigma|],(L_j)_{j\in\sigma}\bigr)^s
\le \frac{\prod_{i=1}^{r}\binom{d-d_i}{d-|\sigma|}\binom{d}{d_i}}{\binom{d}{|\sigma|}^r}\prod_{i=1}^{r}V\bigl(K[d-|\sigma_i|],(L_j)_{j\in\sigma_i}\bigr)
\end{equation}

Let $Z_{j_1}=\sum_{\ell=1}^{m}S_{\ell}$ be a zonotope expressed as a Minkowski sum of line segments $S_1,\dots,S_m$, and fix $j_1\in\sigma$. Since $j_1$ belongs exactly to $s$ of the sets $\sigma_i$, applying \eqref{eq:line-bound} we have 

\begin{align*}
&|K|^{\frac{r-s}{s}}V\big(K[d-|\sigma|],Z_{j_1},(L_j)_{j\in\sigma\setminus\{j_1\}}\big)= \sum_{l=1}^{m}|K|^{\frac{r-s}{s}}V\big(K[d-|\sigma|],S_l,(L_j)_{j\in\sigma\setminus\{j_1\}}\big)\\
&\le \sum_{l=1}^{m}c^{\frac{1}{s}}\prod_{\substack{i\\j_1\notin\sigma_i}}V\big(K[d-|\sigma_i|],(L_j)_{j\in\sigma_i}\big)^{\frac{1}{s}}\prod_{\substack{i\\j_1\in\sigma_i}}V\big(K[d-|\sigma_i|],S_l,(L_j)_{j\in\sigma\setminus\{j_1\}}\big)^{\frac{1}{s}}\\
&= c^{\frac{1}{s}}\prod_{\substack{i\\j_1\notin\sigma_i}}V\big(K[d-|\sigma_i|],(L_j)_{j\in\sigma_i}\big)^{\frac{1}{s}}\Bigg[\sum_{l=1}^{m}\prod_{\substack{i\\j_1\in\sigma_i}}V\big(K[d-|\sigma_i|],S_l,(L_j)_{j\in\sigma\setminus\{j_1\}}\big)^{\frac{1}{s}}\Bigg]\\
&\le c^{\frac{1}{s}}\prod_{\substack{i\\j_1\notin\sigma_i}}V\big(K[d-|\sigma_i|],(L_j)_{j\in\sigma_i}\big)^{\frac{1}{s}}\prod_{\substack{i\\j_1\in\sigma_i}}\bigg[\sum_{l=1}^{m}V\big(K[d-|\sigma_i|],S_l,(L_j)_{j\in\sigma\setminus\{j_1\}}\big)\bigg]^{\frac{1}{s}}\\
&= c^{\frac{1}{s}}\prod_{\substack{i\\j_1\notin\sigma_i}}V\big(K[d-|\sigma_i|],(L_j)_{j\in\sigma_i}\big)^{\frac{1}{s}}\prod_{\substack{i\\j_1\in\sigma_i}}V\big(K[d-|\sigma_i|],Z_{j_i},(L_j)_{j\in\sigma\setminus\{j_1\}}\big)^{\frac{1}{s}},
\end{align*}
where we also used H\"{o}lder's inequality.
Repeating this argument for the other coordinates will give us \eqref{eq:main-mixed} for arbitrary zonotopes $Z_j$.
Using the continuity of mixed volumes, together with the fact that every zonoid is a limit of zonotopes, the result follows.
\end{proof}
\begin{corollary}
    Let $\mathbb{S}_1,\dots,\mathbb{S}_j$ be generalised $d$-hypersurfaces with $1\le j\le d-1$. Suppose $(\sigma_1,\dots,\sigma_r)$ a $s$-uniform cover of $\sigma=\{1,\dots,j\}$ and that $|\sigma_i|=d_i$.Then
    \begin{equation*}
        Q_j^1(\mathbb{S}_1,\dots,\mathbb{S}_j)\le q\prod_{i=1}^rQ_{d_i}^1\big((\mathbb{S}_j)_{j\in\sigma_i}\big)^{\frac{d_i}{sj}},
    \end{equation*}
where
    $$q=\frac{1}{\omega_d^{r/sj}}\bigg(\frac{d!\omega_d}{(d-j)!\omega_{d-j}}\bigg)^{1/j}\Bigg(\prod_{i=1}^r\frac{\binom{j}{d_i}\omega_{d-d_j}}{\binom{d}{d_i}d_i!}\Bigg)^{1/sj}.$$
\end{corollary}

\begin{proof}
    We apply Theorem 3.2 for $\sigma=\{1,\dots,j\}$, $K=B_2^d$ and $Z_i=\Pi(\mathbb{S}_i)$ and we have the result.
\end{proof}

\section{Bounds In The Diagonal Case}\label{section-4}
\noindent In the diagonal case—that is, when all hypersurfaces coincide—the problem is to establish a sharp comparison between
$Q_{j+1}^{p}(\mathbb{S})$ and  $Q_{j}^{p}(\mathbb{S})$. For $p=1,2,\infty$, this was completely resolved in \cite{BrazitikosCarberyMacIntyre2023}, where it was shown that the ratio is maximized by the uniform measure on the sphere.

In this section, we prove that for $0<p<2$ the uniform measure on the sphere also maximizes the ratio
$\dfrac{Q_{2}^{p}(\mathbb{S})}{Q_{1}^{p}(\mathbb{S})}$
among all probability measures on the sphere. 

Moreover, we prove that the maximizer changes when $p>2$ and we prove that, in fact, it is a discrete measure. 

\noindent Let $\mu$ be a probability measure on the sphere $S^{d-1}$. Then, since $Q_{1}^{p}(\mathbb{S})=1$, it suffices to give a sharp upper bound 
for 
\begin{equation}
    \int_{S^{d-1}}\int_{S^{d-1}}|x\wedge y|^pd\mu(x)d\mu(y)=\int_{S^{d-1}}\int_{S^{d-1}}(1-\langle x,y\rangle^2)^{\frac{p}{2}}d\mu(x)d\mu(y).
\end{equation}
The following theorem asserts that this functional is indeed maximized for the uniform measure on the sphere for $0<p<2$, while for $p>2$ the uniform measure is not a maximizer. 
\begin{theorem}
  Let $S^{d-1}\subset\mathbb{R}^d$ be the unit sphere and let $\mu$ be a probability measure on $S^{d-1}$. For $p>0$ define
\[
I_p(\mu):=\iint_{S^{d-1}\times S^{d-1}}(1-\langle x,y\rangle^2)^{p/2}\,d\mu(x)\,d\mu(y).
\]
Then:
\begin{enumerate}
\item For $0<p<2$ the uniform surface measure $\sigma$ on $S^{d-1}$ maximizes $I_p(\mu)$ among probability measures $\mu$.
\item For every $p>2$ the uniform measure is not a maximizer; in fact a discrete measure supported on the four points $\{\pm e_1,\pm e_2\}\subset S^{1}$ gives a strictly larger value.
\end{enumerate}
\end{theorem}
\begin{proof}
For the case $0<p<2$
write $\alpha=\tfrac p2\in(0,1)$ and set $F(t)=(1-t^2)^\alpha$ for $t\in[-1,1]$.
By the generalised binomial theorem,
\[
(1-u)^\alpha = \sum_{k=0}^\infty \binom{\alpha}{k}(-1)^k u^k,\qquad |u|<1,
\]
so with $u=t^2$,
\[
F(t)=1+\sum_{k=1}^\infty c_k t^{2k},\qquad c_k=\binom{\alpha}{k}(-1)^k.
\]
For $\alpha\in(0,1)$ the sign of $\binom{\alpha}{k}$ is $(-1)^{k-1}$ for $k\ge1$, hence $c_k<0$ for all $k\ge1$.

Let $\mu$ be any probability measure on $S^{d-1}$ and let $X,X'$ be independent with law $\mu$. Then,
\[
I_p(\mu)=\mathbb{E}\big[F(\langle X,X'\rangle)\big]
=1+\sum_{k=1}^\infty c_k\,\mathbb{E}\big[\langle X,X'\rangle^{2k}\big].
\]
Thus it suffices to prove that for each $k\ge1$,
\[
\mathbb{E}_\mu\big[\langle X,X'\rangle^{2k}\big]\ge \mathbb{E}_\sigma\big[\langle U,U'\rangle^{2k}\big],
\]
where $U,U'$ are independent uniform on $S^{d-1}$.

Define the $k$-th moment symmetric tensor
\[
M^{(k)}(\mu):=\int_{S^{d-1}} x^{\otimes k}\,d\mu(x)\in\mathrm{Sym}^k(\mathbb{R}^d).
\]
With the natural inner product on symmetric tensors one checks
\[
\|M^{(k)}(\mu)\|^2 = \mathbb{E}\big[\langle X,X'\rangle^{2k}\big].
\]
The orthogonal group $O(d)$ acts on $\mathrm{Sym}^k(\mathbb{R}^d)$ by $R\cdot T=R^{\otimes k}T$, and averaging over $O(d)$ is the orthogonal projection $P_{\mathrm{inv}}$ onto the subspace of $O(d)$-invariant tensors. Consequently for any tensor $T$,
\[
\|T\|^2 = \|P_{\mathrm{inv}}T\|^2 + \|(I-P_{\mathrm{inv}})T\|^2 \ge \|P_{\mathrm{inv}}T\|^2.
\]
For $T=M^{(k)}(\mu)$ the projection $P_{\mathrm{inv}}M^{(k)}(\mu)$ equals the $k$-th moment tensor of the averaged measure $\int_{O(d)}R_*\mu\,dR$, which is the uniform measure $\sigma$. Hence
\[
\|M^{(k)}(\mu)\|^2 \ge \|M^{(k)}(\sigma)\|^2,
\]
as required. 

Now, since each $c_k<0$ and $\|M^{(k)}(\mu)\|^2\ge\|M^{(k)}(\sigma)\|^2$, we get
\[
I_p(\mu)=1+\sum_{k\ge1} c_k\|M^{(k)}(\mu)\|^2 \le 1+\sum_{k\ge1} c_k\|M^{(k)}(\sigma)\|^2 = I_p(\sigma).
\]
Therefore $\sigma$ maximizes $I_p$ when $0<p<2$.

For $p>2$, let $d=2$ and consider
\[
\mu=\tfrac14(\delta_{e_1}+\delta_{-e_1}+\delta_{e_2}+\delta_{-e_2}).
\]
For $X,X'\sim\mu$ independent there are $16$ ordered pairs; exactly $8$ of these are orthogonal pairs giving $|X\wedge X'|^p=1$, and the rest give $0$. Hence
\[
I_p(\mu)=\frac{8}{16}\cdot 1 = \tfrac12.
\]
For the uniform $\sigma$ on $S^1$,
\[
I_p(\sigma) = \frac{1}{\pi}\int_0^\pi \sin^p\theta\,d\theta.
\]
But for $p>2$ one has $\sin^p\theta<\sin^2\theta$ on a set of positive measure, so
\[
I_p(\sigma) < \frac{1}{\pi}\int_0^\pi \sin^2\theta\,d\theta = \tfrac12.
\]
Thus $I_p(\mu)>I_p(\sigma)$ for every $p>2$. The same construction embedded in higher dimensions (take the four points $\pm e_1,\pm e_2\in S^{d-1}$) yields a counterexample for every $d\ge2$.
\end{proof}

\begin{theorem}
Let \(S\) be a \((d-1)\)-dimensional smooth hypersurface in \(\mathbb{R}^d\), 
and let \(\sigma\) be a probability measure on \(S\).
Let \(v:S\to S^{d-1}\) be a measurable vector field taking values on the unit sphere.
For any \(p\ge 2\), define
\[
J_p(v,\sigma)
:=\iint_{S\times S}\big|v(x)\wedge v(y)\big|^p\,d\sigma(x)\,d\sigma(y).
\]
Then
\[
J_p(v,\sigma)\le 1-\frac{1}{d}.
\]
Moreover, equality holds if the push-forward measure
\(\mu:=v_\#\sigma\) (i.e. \(\mu(A)=\sigma(v^{-1}(A))\))
is supported on the set \(\{\pm e_1,\dots,\pm e_d\}\) for some orthonormal basis 
\(\{e_i\}_{i=1}^d\) and satisfies
\(\mu(\{e_i\})+\mu(\{-e_i\})=1/d\) for all \(i=1,\dots,d\).
\end{theorem}

\begin{proof}
Let \(\mu=v_\#\sigma\) be the push-forward measure on \(S^{d-1}\), 
so that for any measurable set \(A\subset S^{d-1}\),
\(\mu(A)=\sigma(v^{-1}(A))\).
Then, by change of variables,
\[
J_p(v,\sigma)
=\iint_{S^{d-1}\times S^{d-1}}\big|u\wedge w\big|^p\,d\mu(u)\,d\mu(w)
=:I_p(\mu).
\]
For \(u,w\in S^{d-1}\), we have
\[
|u\wedge w|^2=1-\langle u,w\rangle^2\in[0,1].
\]
Let \(t=\langle u,w\rangle^2\in[0,1]\) and \(\alpha=p/2\ge1\).
Since \((1-t)^\alpha\le 1-t\) for \(t\in[0,1]\),
it follows that
\[
|u\wedge w|^p=(1-\langle u,w\rangle^2)^\alpha\le 1-\langle u,w\rangle^2.
\]
Integrating this inequality with respect to \(\mu\otimes\mu\) gives
\[
I_p(\mu)\le \iint_{S^{d-1}\times S^{d-1}} (1-\langle u,w\rangle^2)\,d\mu(u)\,d\mu(w)=:I_2(\mu).
\]
Define the moment matrix
\[
M=\int_{S^{d-1}} u\otimes u\,d\mu(u).
\]
Then \(\operatorname{tr}M=1\), and
\[
\iint_{S^{d-1}\times S^{d-1}} \langle u,w\rangle^2\,d\mu(u)\,d\mu(w)
=\|M\|_F^2=\sum_{i=1}^d\lambda_i^2,
\]
where \(\lambda_1,\dots,\lambda_d\) are the eigenvalues of \(M\).
By the Cauchy–Schwarz inequality,
\[
\sum_{i=1}^d\lambda_i^2 \ge \frac{1}{d}\left(\sum_{i=1}^d\lambda_i\right)^2
= \frac{1}{d}.
\]
Hence,
\[
I_2(\mu)=1-\sum_{i=1}^d\lambda_i^2 \le 1-\frac{1}{d}.
\]
Combining inequalities,
\[
J_p(v,\sigma)=I_p(\mu)\le I_2(\mu)\le 1-\frac{1}{d}.
\]
\textit{Equality conditions.}
Equality holds if and only if:
\begin{itemize}
  \item[(i)] \((1-\langle u,w\rangle^2)^\alpha=1-\langle u,w\rangle^2\) 
  for \(\mu\otimes\mu\)-almost every \((u,w)\).
  For \(\alpha>1\), this occurs only when 
  \(\langle u,w\rangle^2\in\{0,1\}\) almost everywhere.
  (For \(p=2\), this inequality is an equality for all \(u,w\).)
  \item[(ii)] \(\sum_i\lambda_i^2=1/d\), i.e.\ all eigenvalues of \(M\) are equal to \(1/d\).
  This is equivalent to \(M=\tfrac{1}{d}I_d\),
  which happens precisely when \(\mu\) is distributed uniformly over an orthonormal basis and its negatives,
  with \(\mu(\{e_i\})+\mu(\{-e_i\})=1/d\) for each \(i\).
\end{itemize}
This completes the proof.
\end{proof}

\paragraph{Remark.}
For \(p=2\), the inequality step \((1-t)^{p/2}\le1-t\) is an equality for all \(t\),
so all probability measures \(\mu\) with \(M=\tfrac{1}{d}I_d\) achieve the bound.
For \(p>2\), equality additionally requires that 
\(\langle u,w\rangle^2\in\{0,1\}\) \(\mu\otimes\mu\)-almost everywhere,
which forces \(\mu\) to be discrete as above.

\section{Sharp Estimates for Visibility}\label{section-5}
\noindent We start by recalling the principal definitions and objects that will be used throughout this section. The {\bf projection body} of a generalised $d$-hypersurface $\mathbb {S}=(S,\sigma,v)$ is the convex body $\Pi(\mathbb{S})$  with support function
\begin{equation*}
    h_{\Pi(\mathbb S)}(y)=\int_{S}|\langle y,v(x)\rangle|d\sigma(x).
\end{equation*}
We will denote its polar as 
\begin{equation*}
    K(\mathbb S)=\Pi(\mathbb S)^{\circ}
\end{equation*}
and the {\bf visibility} of $\mathbb{S}$ by
\begin{equation}
    \mathrm{vis}(\mathbb{S})= |K(\mathbb{S})|^{-1/d}.
\end{equation}
More generally, for $1\le p<\infty$ we can define the covex body $K^p(\mathbb{S})$ as the body with norm
\begin{equation}
\|y\|_{K^p(\mathbb{S})}=\bigg(\int_{S}|\langle y,v(x)\rangle|^pd\sigma(x)\bigg)^{1/p}
\end{equation}
and through this the {\bf p-visibility} of $\mathbb{S}$ as
\begin{equation}
\mathrm{vis}^p(\mathbb{S})=|K^p(\mathbb{S})|^{-1/d}.
\end{equation}
One of the central ideas in the works of Guth \cite{Guth} and Zhang \cite{Zhang} was to factorize $\mathrm{vis}(\mathbb{S})$—that is, to express it, up to absolute constants, as a product whose factors are the quantities $$
\sigma(E,\mathbb{S})=\int_{\mathbb{S}}\dots\int_{\mathbb{S}}|E^\perp\wedge\upsilon(x_1)\wedge\dots\wedge\upsilon(x_d)|d\sigma(x_1)\dots d\sigma(x_d).$$ 
In what follows, we provide a sharp comparison between the visibility and the product.
\subsection{Case $p=1$.}
We start with a crucial identity that was proved in \cite[Theorem 3.7]{BrazitikosCarberyMacIntyre2023}.
\begin{equation}\label{wedge_product_identity}
\int_{\mathbb{S}}\dots\int_{\mathbb{S}}|\upsilon(x_1)\wedge\dots\dots\wedge\upsilon(x_d)|d\sigma(x_1)\dots d\sigma(x_d)=\frac{d!}{2^d}|\Pi(\mathbb{S})|
\end{equation}
The first main step in \cite[Theorem 3.1]{Zhang} asserts that under the assumption $h_{\Pi(\mathbb S)}(y)\geq 1$, we have 
$$\int_{\mathbb{S}}\dots\int_{\mathbb{S}}|\upsilon(x_1)\wedge\dots\dots\wedge\upsilon(x_d)|d\sigma(x_1)\dots d\sigma(x_d)\gtrsim \mathrm{vis}(\mathbb{S})^d. $$ Using the above crucial identity and the definition of visibility, we observe that the above Theorem is equivalent to 
$$|\Pi(\mathbb{S})||\Pi(\mathbb{S})^{\circ}|\gtrsim 1.$$ Note that the last one is the so-called reverse Santal\'{o} inequality. For general convex bodies it was conjectured by Mahler that for all symmetric convex bodies $K\subseteq\mathbb{R}^n$, 
$$|K||K^{\circ}|\geq\frac{4^n}{n!}.$$ Note that we obtain the lower bound for all images of the cube. Bourgain and Milman first proved in \cite{BourgainMilman1987} that the above inequality holds for an absolute constant $c$ instead of $4$. 
However, Gordon, Meyer and Reisner proved in \cite{Reisnzer-Gordon-Meyer} that Mahler's conjecture is true for all convex bodies that are zonoids.

Note that $\Pi(\mathbb{S})$ is a convex and compact set. Moreover, the assumption $h_{\Pi(\mathbb S)}(y)\geq 1$ implies that $\Pi(\mathbb{S})$ contains the unit Euclidean Ball. Therefore, it has non-empty interior which means that it is a convex body. From the above discussion we end up with the following theorem.
\begin{theorem}\label{vis_santalo}
Let $\mathbb{S}$ be a hypersurface. If $\Pi(\mathbb{S})$ has non-empty interior, then
$$\int_{\mathbb{S}}\dots\int_{\mathbb{S}}|\upsilon(x_1)\wedge\dots\dots\wedge\upsilon(x_d)|d\sigma(x_1)\dots d\sigma(x_d)\geq  (2\mathrm{vis}(\mathbb{S}))^d.$$
Moreover, we have equality if and only if $\Pi(\mathbb{S})=T(B_{\infty}^d)$ and $T\in GL(d,\mathbb{R})$.
\end{theorem}

\noindent A similar identity to \eqref{wedge_product_identity} was proved in \cite{BrazitikosCarberyMacIntyre2023},  regarding ``the factors of visibility".
\begin{align}\label{sigma_proj}
&\sigma(E,\mathbb{S})=\int_{\mathbb{S}}\dots\int_{\mathbb{S}}|E^\perp\wedge\upsilon(x_1)\wedge\dots\wedge\upsilon(x_d)|d\sigma(x_1)\dots d\sigma(x_d)\\
\nonumber&=\int_{\mathbb{S}}\dots\int_{\mathbb{S}}|P_E(\upsilon(x_1))\wedge\dots\wedge P_E(\upsilon(x_d))|d\sigma(x_1)\dots d\sigma(x_d)\\
\nonumber&=\frac{k_i!}{2^{k_i}}|P_E(\Pi(\mathbb{S}))|.
\end{align}

On the one hand, to give an upper bound for visibility, we will use the following affine Loomis--Whitney inequality from \cite{GBBC}.

\begin{theorem}[Affine Loomis--Whitney inequality]
    Let $\{w_i\}_{i=1}^{d}$ be a basis of $\mathbb{R}^d$ and let $(\sigma_1,\dots, \sigma_m)$ be a uniform cover of $[d]$ with weights $(p_1,\dots ,p_m)$. Let $H_i = span\{w_j : j\in\sigma_i\}$, and $p =\sum_{i=1}^{m}p_i$. Then, for every compact $K\subseteq\mathbb{R}^d$
we have
\begin{equation*}
    |K|\le BL_2\prod_{i=1}^{m}|P_{H_{i}}(K)|^{p_i},
\end{equation*}
 where $BL_2:=\dfrac{\prod_{i=1}^{m}|\wedge_{k\in\sigma_{i}}w_{k}|^{p_{i}}}{|\wedge_{i=1}^{d}w_{i}|}$.
\end{theorem}
On the other hand, to give a lower bound we will use the reverse dual Loomis--Whitney inequality from \cite{alonso2021reverse}.

\begin{theorem}[Reverse dual Loomis Whitney]\label{reverse dual}
There exists an absolute constant $C > 0$, such that for every centered convex body $K\in\mathbb{R}^d$, there exists an orthonormal basis $\{w_i\}_{i=1}^{d}$ of $\mathbb{R}^d$ such that for any uniform cover $(\sigma_1,\dots, \sigma_m)$ of $[d]$ with weights $(p_1,\dots ,p_m)$, $p =\sum_{i=1}^{m}p_i$ and $H_i = span\{w_j : j\in\sigma_i\}$ we have
\begin{equation*}
     C^{(p-1)d}|K|\le\prod_{i=1}^m|K\cap H_{i}|^{p_{i}}.
\end{equation*}
\end{theorem}
We are now in position to formulate and prove the following theorem.
\begin{theorem}
Let $d_1+\dots+d_m=d$. Then,
\begin{equation*}
    a_d\inf_{\substack{E_j\in G_{d,d_j}}}\Bigg(\frac{\sigma(E_1,\mathbb{S})\dots\sigma(E_m,\mathbb{S})}{|E_1\wedge\dots\wedge E_m|}\Bigg)^{1/d}\le \mathrm{vis}(\mathbb{S})\le b_d\inf_{\substack{E_j\in G_{d,d_j}}}\Bigg(\frac{\sigma(E_1,\mathbb{S})\dots\sigma(E_m,\mathbb{S})}{|E_1\wedge\dots\wedge E_m|}\Bigg)^{1/d}
\end{equation*}
with $a_d=\displaystyle \Big(\frac{2^dC^{(m-1)d}}{\prod_{i=1}^{m}\omega_{i}^{2}\prod_{i=1}^{m}d_{i}!}\Big)^{1/d}$ and $\, b_d=\displaystyle\Big(\frac{d!}{2^d\prod_{i=1}^{m}d_i!}\Big)^{1/d}$.
\end{theorem}

\begin{proof}
Let $\{w_i\}_{i=1}^{d}$ be a basis of $\mathbb{R}^n$, let $(\sigma_1,\dots, \sigma_m)$ be a 1-uniform cover of $[d]$ with $|\sigma_i|=k_i$, and for $i=1,\dots,m$, set $E_{i}=\mathrm{span}\{w_j : j\in\sigma_i\}$. Then, $k_1+\dots+k_m=d$ and  $|\wedge_{i=1}^{n}w_i|=|E_1\wedge\dots\wedge E_m|$.
For the upper bound, we successively apply the reverse Santaló inequality and the affine Loomis–Whitney inequality, taking $H_{i}=E_{i}$ and $p_{i}=1$ for $i=1,\dots,m$ so that
    \begin{align*}
        &vis(\mathbb{S})^d=\frac{1}{|K(\mathbb{S})|}
        \le \frac{d!}{4^{d}}|\Pi(\mathbb{S})|\\
        &\le\frac{d!}{4^{d}}BL_{2}\prod_{i=1}^{m}|P_{E_{i}}(\Pi(\mathbb{S}))|\\
        &=\frac{d!}{4^{d}}\frac{\prod_{i=1}^{m}|\wedge_{j\in\sigma_{i}}w_{i}|}{|\wedge_{i=1}^{n}w_{i}|}\prod_{i=1}^m\frac{2^{d_i}}{d_i!}\sigma(E_i,\mathbb{S})\\
        &=\frac{d!}{2^d\prod_{i=1}^{m}d_i!}\prod_{i=1}^{m}|\wedge_{j\in\sigma_{i}}w_{i}|\frac{\sigma(E_1,\mathbb{S})\dots\sigma(E_m,\mathbb{S})}{|E_1\wedge\dots\wedge E_m|}.
    \end{align*}
Since $|\wedge_{j\in\sigma_{i}}w_{i}|\le 1$, we arrive at the desired upper bound 
\begin{equation}
    vis(\mathbb{S})\le\Bigg(\frac{d!}{2^d\prod_{i=1}^{m}d_i!}\Bigg)^{1/d}\Bigg(\frac{\sigma(E_1,\mathbb{S})\dots\sigma(E_m,\mathbb{S})}{|E_1\wedge\dots\wedge E_m|}\Bigg)^{1/d}.
\end{equation}

For the lower bound, we select an orthonormal basis $\{u_i\}_{i=1}^{d}$, as prescribed by the reverse dual Loomis–Whitney inequality applied to the body $K(\mathbb{S})$. Setting $E_{i}= \mathrm{span}\{u_j : j\in\sigma_i\}$, where $(\sigma_1,\dots, \sigma_m)$ is the 1-uniform cover of $[d]$ as above, we obtain
\begin{align*}
&vis(\mathbb{S})^d=\frac{1}{|K(\mathbb{S})|}\\
&\ge\frac{C^{(m-1)d}}{\prod_{i=1}^{m}|K(\mathbb{S})\cap E_i|}\\
&\ge\frac{C^{(m-1)d}}{\prod_{i=1}^{m}\omega_{i}^{2}}\prod_{i=1}^{m}|P_{E_{i}}(\Pi(\mathbb{S}))|\\
&=\frac{C^{(m-1)d}\prod_{i=1}^{m}2^{d_i}}{\prod_{i=1}^{m}\omega_{i}^{2}\prod_{i=1}^{m}d_{i}!}\prod_{i=1}^{m}\sigma(E_i,\mathbb{S})\\
&=\frac{2^dC^{(m-1)d}}{\prod_{i=1}^{m}\omega_{i}^{2}\prod_{i=1}^{m}d_{i}!}\frac{\sigma(E_1,\mathbb{S})\dots\sigma(E_m,\mathbb{S})}{|E_1\wedge\dots\wedge E_m|}
\end{align*}
and we have
\begin{equation*}
   \Bigg(\frac{2^dC^{(m-1)d}}{\prod_{i=1}^{m}\omega_{i}^{2}\prod_{i=1}^{m}d_{i}!}\Bigg)^{1/d}\Bigg(\frac{\sigma(E_1,\mathbb{S})\dots\sigma(E_m,\mathbb{S})}{|E_1\wedge\dots\wedge E_m|}\Bigg)^{1/d}\le vis(\mathbb{S}),
\end{equation*}
as promised.
\end{proof}
A closer examination of the proof of Theorem \ref{reverse dual} reveals that the subspaces which reverse the dual Loomis--Whitney inequality were derived from an orthonormal basis that renders the linear map placing the body in the isotropic position a diagonal matrix. For a sharp lower bound we can take a different route by choosing the orthonormal basis from the Lewis position. 

\begin{proposition}[Lewis position]\label{lewis_lemma}
Let $(S,\sigma)$ be a finite measure space and $v:S\to\mathbb R^d$ measurable with 
\[
\int_S\|v(x)\|_2^p\,d\sigma(x)<\infty
\qquad(1< p<\infty).
\]
Define for $A\in L(\mathbb R^d)$ the functional
\[
a(A):=\Big(\int_S\|A v(x)\|_2^p\,d\sigma(x)\Big)^{1/p}
\]
Then there exists an invertible matrix $u\in GL(d)$ such that
\[
a(u)=1\qquad\text{and}\qquad a^*(u^{-1})=d,
\]
where $a^*$ is the dual norm on $L(\mathbb R^d)$ defined by $a^*(B)=\sup_{a(A)\le1}|\operatorname{tr}(B^T A)|$.
Moreover, for such $u$ the following isotropic identity holds
\begin{equation}
\int_S \|u v(x)\|_2^{p-2}\,(u v(x))\, \otimes (u v(x)) \,d\sigma(x) \;=\; \frac{1}{d}\,I_d.
\end{equation}
Additionally, the above also holds for $p=1$, provided that the vector field $v$ satisfies some extra assumptions, namely
\[
v\in L^1(S;\mathbb R^d),\qquad \sigma(\{x:\,v(x)=0\})=0,
\qquad \text{ess.\ span}\{v(x):x\in S\}=\mathbb R^d.
\]
\end{proposition}
The proof of the above can be found in the Appendix. 
The isotropic identity proves useful in conjunction with the continuous Brascamp--Lieb inequality. In our case, we require the following refinement of the one proved in \cite{Barthe2004} (see also \cite{Brazitikos-Giannopoulos_BL} for its approximate form).
\begin{theorem}[Continuous Brascamp--Lieb on $S$]\label{cont_brascamp_lieb}
Let $S$ be a hypersurface equipped with a positive measure $\mu$.
Let $v:S\to\mathbb{R}^n$ be measurable with $v(x)\neq0$ for $\mu$‑almost every $x\in S$, and assume the normalization
\begin{equation}\label{eq:identity_assumption}
I_d= \int_S \frac{v(x)\otimes v(x)}{\|v(x)\|_2^2}\,d\mu(x).
\end{equation}
Define the unit directions $u(x):=v(x)/\|v(x)\|_2\in S^{n-1}$ and suppose we are given a measurable family
of non-negative functions ${f_x:\mathbb R\to[0,\infty)},{x\in S}$ that meets the usual technical integrability
conditions.

Then the following inequality holds:
\begin{equation}\label{eq:conclusion_S}
\int_{\mathbb R^n}\exp\Big(\int_S\log\bigl(f_x(\langle y,u(x)\rangle)\bigr)\,d\mu(x)\Big)\,dy
\le
\exp\Big(\int_S\log\Big(\int_{\mathbb R} f_x(t)\,dt\Big)\, d\mu(x)\Big).
\end{equation}
\end{theorem}

Starting with an arbitrary $\mathbb{S}=(S,v,\mu)$, we find $u$ according to Proposition \ref{lewis_lemma} and let $w_1,\ldots, w_d$ be an orthonormal basis that renders $u$ a diagonal map. In other words, for all $i=1,\ldots, d$ we have 
$$u(w_i)=\lambda_iw_i,$$ for some $\lambda_i$. Moreover, we obtain a new vector field $z=u v$. The relation to the original field $v$ arises from the linear equivariance of $K(\mathbb{S}_v)$; one has $K(\mathbb{S}_z)=u^{-T}K(\mathbb{S}_v)$, hence
\[
\mathrm{vis}(\mathbb S)=|\det u|^{-1/d}\,\mathrm{vis}(\mathbb{S}_z).
\]
We can express the volume of $K(\mathbb{S}_z)$ using the formula $$d!|K(\mathbb{S}_z)|=\int_{\mathbb{R}^d}e^{-\|y\|_{K(\mathbb{S}_z)}}dx.$$
Set $d\mu(x):=d\|uv(x)\|_2\, d\sigma(x)$. The last integral can now be bounded using the continuous Brascamp--Lieb inequality on $S$ as Theorem \ref{cont_brascamp_lieb} for the function $f_x(t)=e^{-|t|/d}$ as follows:
    \begin{align*}
    \int_{\mathbb{R}^d}e^{-\|y\|_{K(\mathbb{S}_z)}}dy&=\int_{\mathbb{R}^d}\exp\bigg\{-\int_{S}|\langle y,uv(x)\rangle|d\sigma(x)\bigg\}dy\\ &=\int_{\mathbb{R}^d}\exp\bigg\{-\int_{S}\left|\left\langle y,\frac{uv(x)}{\|uv(x)\|_2}\right\rangle\right|\, \|uv(x)\|_2d\sigma(x)\bigg\}dy\\        &=\int_{\mathbb{R}^d}\exp\bigg\{\int_{S}\log f_x(\langle y,uv(x)\rangle)d\mu(x)\bigg\}dy\\
        &\le \exp\bigg\{\int_{S}\log\Big(\int_{\mathbb{R}}f_x(t)dt\Big)d\mu(x)\bigg\}\\
        &=\exp\bigg\{\int_{S}\log (2d)\, d\mu(x)\bigg\}\\
        &=(2d)^d,
    \end{align*}
since $$ d=\mathrm{tr}\left(I_d\right)=\int_{S}\mathrm{tr}\big( s(x)\otimes s(x)\big)d\mu(x)=\int_{S}\|s(x)\|_2^2\, d\mu(x)=\int_{S}1\, d\mu(x)=d\int_{S}\|uv(x)\|_2 \,  d\sigma(x).$$ Combining the above yields 
\begin{equation}\label{vis_z_bound}
\mathrm{vis}(\mathbb{S}_z)\geq\frac{(d!)^{1/d}}{2d}.
\end{equation}
Now, we need an upper bound for the product of $\sigma(E_i,\mathbb{S})$. First, we pass to the vector field $z$.
\begin{align*}
&\int_S\dots\int_S|P_{E_i}(v(x_1))\wedge\dots\wedge P_{E_i}(v(x_{k_i}))|d\sigma(x_1)\dots d\sigma(x_{k_i})=\\
&\frac{1}{\prod_{j\in\sigma_i}\lambda_j}\int_S\dots\int_S|P_{E_i}(uv(x_1))\wedge\dots\wedge P_{E_i}(uv(x_{k_i}))|d\sigma(x_1)\dots d\sigma(x_{k_i}).
\end{align*} 
Using the Cauchy-Schwarz inequality, we obtain 
\begin{align*}
&\frac{1}{\prod_{j\in\sigma_i}\lambda_j}\int_S\dots\int_S \Bigg|P_{E_i}\bigg(\frac{z(x_1)}{\|z(x_1)\|_2}\bigg)\wedge\dots\wedge P_{E_i}\bigg(\frac{z(x_{k_i})}{\|z(x_{k_i})\|_2}\bigg)\Bigg|\prod_{j=1}^{k_i}\|z(x_j)\|_2d\sigma(x_1)\dots d\sigma(x_{k_i})\\
   &\le\frac{1}{\prod_{j\in\sigma_i}\lambda_j}\Bigg[\int_S\dots\int_S\Bigg|P_{E_i}\bigg(\frac{z(x_1)}{\|z(x_1)\|_2}\bigg)\wedge\dots\wedge P_{E_i}\bigg(\frac{z(x_{k_i})}{\|z(x_{k_i})\|_2}\bigg)\Bigg|^2\prod_{j=1}^{k_i}\|z(x_j)\|_2d\sigma(x_1)\dots d\sigma(x_{k_i})\Bigg]^{1/2}\\
    &\times\Bigg[\int_S\dots\int_S\prod_{j=1}^{k_i}\|z(x_j)\|_2d\sigma(x_1)\dots d\sigma(x_{k_i})\Bigg]^{1/2}\\
   &=\frac{1}{\prod_{j\in\sigma_i}\lambda_j}\left(\frac{k_i!}{d^{k_i}}\right)^{1/2}.
\end{align*}
Therefore, using that $\sigma$ is an $1$-uniform cover and \eqref{vis_z_bound}  multiplication gives 
\begin{equation}\label{reverse_projections}
\frac{\sigma(E_1,\mathbb{S})\dots\sigma(E_m,\mathbb{S})}{|E_1\wedge\dots\wedge E_m|}\leq\frac{\displaystyle\prod_{i=1}^m \left(\frac{k_i!}{d^{k_i}}\right)^{1/2}}{{|w_1\wedge\dots\wedge w_d|\prod_{j\in [d]}\lambda_j}}=\frac{\displaystyle\prod_{i=1}^m \left(\frac{k_i!}{d^{k_i}}\right)^{1/2}}{|\det u|}\leq \frac{2^d\, d^{d/2}\, \mathrm{vis}(\mathbb{S})^d\displaystyle\prod_{i=1}^m \sqrt{k_i!}}{d!}.
\end{equation}
We therefore obtain the following theorem.
\begin{theorem}
Let $d_1+\dots+d_m=d$. Then,
\begin{equation*}
    c_d\inf_{\substack{E_j\in G_{d,d_j}}}\Bigg(\frac{\sigma(E_1,\mathbb{S})\dots\sigma(E_m,\mathbb{S})}{|E_1\wedge\dots\wedge E_m|}\Bigg)^{1/d}\le \mathrm{vis}(\mathbb{S})\le b_d\inf_{\substack{E_j\in G_{d,d_j}}}\Bigg(\frac{\sigma(E_1,\mathbb{S})\dots\sigma(E_m,\mathbb{S})}{|E_1\wedge\dots\wedge E_m|}\Bigg)^{1/d}
\end{equation*}
with $c_d=\displaystyle \frac{1}{2\sqrt{d}}\Big(\frac{d!}{\prod_{i=1}^{m}\sqrt{d_{i}!}}\Big)^{1/d}$ and $\, b_d=\displaystyle\frac{1}{2}\Big(\frac{d!}{\prod_{i=1}^{m}d_i!}\Big)^{1/d}$.
\end{theorem}
\bigskip 

We close this section with a reverse Loomis--Whitney inequality for $\Pi(\mathbb{S})$ and, in general, for all zonoids. This was first obtained for all convex bodies $K$ (but when the dimension of subspaces is $n-1$) in \cite{Campi2018} and later the authors \cite{KoldobskySaroglouZvavitch2019} obtained an optimal dependence on the dimension.

We will work as before with the linear map $u$ and the induced orthonormal basis $w_i$. Having already an upper bound for the volume of projections, we only need a lower bound for the volume. 

\begin{align*}
        &|\det u|\, |\Pi(\mathbb{S})|=\frac{2^d}{d!}\int_S\dots\int_S|uv(x_1)\wedge\dots\wedge uv(x_d)|d\sigma(x_1)\dots d\sigma(x_d)\\
        &=\frac{2^d}{d!}\int_S\dots\int_S\Bigg|\frac{uv(x_1)}{\|uv(x_1)\|_2}\wedge\dots\wedge\frac{uv(x_d)}{\|uv(x_d)\|_2}\Bigg|\prod_{j=1}^d\|uv(x_j)\|_2d\sigma(x_1)\dots d\sigma(x_d)\\
        &\ge\frac{2^d}{d!}\int_S\dots\int_S\Bigg|\frac{uv(x_1)}{\|uv(x_1)\|_2}\wedge\dots\wedge\frac{uv(x_d)}{\|uv(x_d)\|_2}\Bigg|^2\prod_{j=1}^d\|uv(x_j)\|_2d\sigma(x_1)\dots d\sigma(x_d)\\
        &=\left(\frac{2}{d}\right)^d.
    \end{align*}
Using \eqref{sigma_proj} and \eqref{reverse_projections}, we obtain the following reverse Loomis--Whitney inequality. 
\begin{theorem}\label{reverse_loomis_zonoids}
    Let $\mathbb{S}=(S,\sigma,v)$ a generalised $d$-hypersurface. There exists an orthonormal basis $\{w_i\}_{i=1}^{d}$ of $\mathbb{R}^d$ such that for any $1$-uniform cover $(\sigma_1,\dots, \sigma_m)$ of $[d]$ and $E_i = span\{w_j : j\in\sigma_i\}$ we have
     \begin{equation}
    \prod_{j=1}^m|P_{E_j}(\Pi(\mathbb{S}))|\le\frac{(\sqrt{d})^{d}}{\prod_{j=1}^m\sqrt{d_j!}}|\Pi(\mathbb{S})|.
     \end{equation}
\end{theorem}
If we take $S=S^{d-1}$, $v(x)=x$, and $\mu$ an isotropic measure on the sphere and if $Z$ is the zonoid generated from that measure, the above yields a reverse Loomis--Whitney inequality for all zonoids.
\begin{corollary}\label{zonoids_reverse}
 For every zonoid $Z\subseteq\mathbb{R}^d$  there exists an orthonormal basis $\{w_i\}_{i=1}^{d}$ of $\mathbb{R}^d$ such that for any $1$-uniform cover $(\sigma_1,\dots, \sigma_m)$ of $[d]$ and $E_i =\mathrm{span}\{w_j : j\in\sigma_i\}$ we have
    \begin{equation*}
        \prod_{i=1}^m|P_{E_i}(Z)|\le\frac{(\sqrt{d})^{d}}{\prod_{i=1}^m\sqrt{d_i!}}|Z|.
    \end{equation*}
    \end{corollary}
   \begin{remark} 
Let $K$ be a convex body in $\R^d$ and consider the projection body $\Pi(K)$ with support function $h_{\Pi(K)}(u)=|P_{u^{\perp}}(K)|=\frac{1}{2}\int_{S^{d-1}}|\langle \theta,u\rangle|\sigma_{K}(\theta)$ where $\sigma_{K}$ is the area measure of $K$. Then, Corollary \ref{zonoids_reverse}   provides an orthonormal basis $\{u_1,\dots,u_d\}$ of $\mathbb{R}^d$, such that 
    \begin{equation}
    \prod_{i=1}^d|P_{\langle u_i\rangle}(K)|\le(\sqrt{d})^d\, |\Pi(K)|.
    \end{equation}
    Noticing that $|P_{\langle u_i\rangle}(\Pi(K))|=2h_{\Pi(K)}(u_i)=2|P_{u_i^{\perp}}(K)|$ and using successively the Santal\'{o} inequality $|\Pi(K)||\Pi^*(K)|\le\omega_2^d$ and Zhang's inequality $\binom{2d}{d}/d^d\le|K|^{d-1}|\Pi^*(K)|$ we conclude that
    \begin{equation} \prod_{i=1}^d|P_{u_i^{\perp}}(K)|\le\frac{d^d\omega_d^2}{2^d\binom{2d}{d}}(\sqrt{d})^d|K|^{d-1}\le(c\sqrt{d})^d|K|^{d-1},
     \end{equation}
recovering the reverse Loomis--Whitney inequality from \cite{KoldobskySaroglouZvavitch2019}. 
\end{remark}
\subsection{Case  p=2.} In this subsection we will prove analogous results with the case $p=1$.
 Recall that the convex body $K^2(\mathbb{S})$ has the norm $$\|y\|_{K^2(\mathbb{S})}=\bigg(\int_S|\langle y,v(x)\rangle|^2d\sigma(x)\bigg)^{1/2},$$ which means it is an ellipsoid. In order to provide bounds for $\mathrm{vis}^2(\mathbb{S})$with respect to $\sigma(E,\cdot)$, where $E\in G_{d,k}$, we first establish a sharp Loomis--Whitney inequality and its reverse counterpart for ellipsoids. The subsequent theorem of \cite{GBBC} is essential for this endeavor.
\begin{theorem}
    Let $\{w_1,\dots,w_d\}$ be a basis of $\mathbb{R}^d$, let $m\ge1$ and let $(\sigma_1,\dots,\sigma_m)$ be a uniform cover of $[d]$ with weights $(p_1,\dots,p_m)$. Let $E_j=span\{w_k:k\in\sigma_j\}$ and $dim(E_j)=|\sigma_j|=d_j$ then for all non-negative integrable functions $f_j:E_j\longrightarrow\mathbb{R}$ we have
\begin{equation}    \int_{\mathbb{R}^d}\prod_{j=1}^mf_j\big(P_{E_j}(x)\big)dx\le \frac{\prod_{j=1}^m|\wedge_{k\in\sigma_j}w_k|_2^{p_j}}{|\wedge_{j=1}^dw_j|}\prod_{j=1}^m\bigg(\int_{E_j}f_j(x)dx\bigg)^{p_j}.
\end{equation}
Moreover, if $f$ is a log-concave function
    \begin{equation}\label{revers_log_conc}       \int_{\mathbb{R}^d}f^d(x)dx\ge\frac{1}{BL_2}\frac{\prod_{j=1}^md_j^{p_jd_j}}{d^d}\prod_{j=1}^m\Bigg(\int_{E_j}f^{d_j}(x)dx\Bigg)^{p_j}.
    \end{equation}
\end{theorem}
\bigskip
\begin{theorem}[Ellipsoids and Brascamp--Lieb]\label{ell_brascamp_lieb}
Let $\mathcal{E}\subset\mathbb{R}^d$ be an ellipsoid and $\{w_1,\dots,w_d\}$ a basis of $\mathbb{R}^d$. Let $m\ge1$ and let $(\sigma_1,\dots,\sigma_m)$ a uniform cover of $[d]$ with weights $(p_1,\dots,p_m)$. Let $F_j=span\{w_k:k\in\sigma_j\}$ and $dim(F_j)=|\sigma_j|=d_j$ let $\{P_{F_j}\}$ be orthogonal projections onto subspaces $F_j$ with $\dim F_j=d_j$.
Then
\begin{equation}\label{main-ell}
\frac{c_{\mathrm{ell}}}{BL_2}\prod_j |P_{F_j}\mathcal{E}|^{p_j}\le|\mathcal{E}| \le C_{\mathrm{ell}}BL_2\prod_j |P_{F_j}\mathcal{E}|^{p_j},
\end{equation}
where
\begin{equation}\label{Cell}
C_{\mathrm{ell}}\;=\;\frac{\omega_d}{\prod_j \omega_{d_j}^{\,p_j}} \qquad\text{and }\qquad\text{} c_{\mathrm{ell}}=\frac{\omega_d}{d^{d/2}}\prod_{j=1}^{m}\Bigg(\frac{d_j^{d_j/2}}{\omega_{d_j}}\Bigg)^{p_j}
\end{equation}
and $\omega_m=|B_2^m|=\pi^{m/2}/\Gamma(\tfrac m2+1)$ is the volume of the unit Euclidean ball in $\mathbb{R}^m$. Moreover, the constant $C_{\mathrm{ell}}$ is optimal (it cannot be decreased uniformly for all ellipsoids).
\end{theorem}

\begin{proof}
Write the ellipsoid in the quadratic form representation
\[
\mathcal{E}=\{x\in\mathbb{R}^d:\;x^T M^{-1} x\le 1\},
\]
with $M$ a symmetric positive definite $d\times d$ matrix. Choose $A$ such that $M=AA^T$. Then $\mathcal{E}=A(B_2^d)$ where $B_2^d$ is the unit Euclidean ball in $\mathbb{R}^d$. Consequently
\begin{equation}\label{vol-E}
|\mathcal{E}|=\omega_d\,\det(A)=\omega_d\,\sqrt{\det M}.
\end{equation}
For each projection $P_{F_j}$ the image $P_{F_j}\mathcal{E}$ equals $(P_{F_j}A)(B_2^d)$ (up to identification with $F_j$), hence
\begin{equation}\label{vol-proj}
|P_{F_j}\mathcal{E}|=\omega_{d_j}\,\det(P_{F_j}A)=\omega_{d_j}\,\sqrt{\det\big(P_{F_j} M P_{F_j}^T\big)}.
\end{equation}
Combine \eqref{vol-E} and \eqref{vol-proj} to write the ratio appearing in \eqref{main-ell}:
\[
\frac{|\mathcal{E}|}{\prod_j|P_{F_j}\mathcal{E}|^{p_j}} = \frac{\omega_d}{\prod_j \omega_{d_j}^{\,p_j}}\;\cdot\; \frac{(\det M)^{1/2}}{\prod_j\big(\det(P_{F_j} M P_{F_j}^T)\big)^{p_j/2}}.
\]
Thus to prove \eqref{main-ell} with the constant \eqref{Cell} it suffices to show
\begin{equation}\label{det-ineq}
\det M^{1/2} \le BL_2\prod_j \det\big(P_{F_j} M P_{F_j}^T\big)^{p_j/2}\qquad\text{for every }M>0.
\end{equation}
The determinant inequality \eqref{det-ineq} is precisely the Gaussian (matrix) formulation of the Brascamp--Lieb inequality in this orthogonal-projection setting. Equivalently, one may deduce \eqref{det-ineq} by evaluating the functional Brascamp--Lieb inequality on centered Gaussian functions. Concretely, for any symmetric positive definite covariance matrix $M$ consider the Gaussian density on $\mathbb{R}^d$ with covariance matrix $M$; evaluating the Brascamp--Lieb functional at the marginal Gaussians, using the matrix identity $\sum_j p_jP_{F_j}=I_d$ and the formula $\int_{\mathbb{R}^d}e^{-\pi\langle{M^{-1}x,x}\rangle}dx=\det(M)^{1/2}$ yields the inequality
\begin{equation*}
    \int_{\mathbb{R}^d} e^{-\pi\langle{M^{-1}x,x}\rangle}\,dx \le BL_2\prod_j \Bigg(\int_{F_j} e^{-\pi\langle{(P_{F_j} M P_{F_j}^T)^{-1}y,y}\rangle} dy\Bigg)^{p_j}.
\end{equation*}
Computing these Gaussian integrals explicitly gives
\begin{equation*}
\det M^{1/2} \le BL_2\prod_j \det\big(P_{F_j} M P_{F_j}^T\big)^{p_j/2}.
\end{equation*}
Inserting \eqref{det-ineq} into the ratio we obtain
\[
\frac{|\mathcal{E}|}{\prod_j|P_{F_j}\mathcal{E}|^{p_j}} \le \frac{\omega_d}{\prod_j \omega_{d_j}^{\,p_j}}BL_2,
\]
which is exactly \eqref{main-ell}. The constant $C_{\mathrm{ell}}$ is optimal because equality holds for Gaussian (and hence ellipsoidal) examples for which $M$ is a suitably aligned scalar multiple of the identity. In particular, if $A$ is a diagonal matrix with the coordinate axes aligned with the $F_j$'s in the natural way, then equality is achieved up to the normalization constants $\omega_m$.

For the left hand inequality, we notice that 
$$\langle M^{-1}x,x\rangle=\langle (AA^T)^{-1}x,x\rangle=\langle A^{-T}A^{-1}x,x\rangle=\langle A^{-1}x,A^{-1}x\rangle=\|A^{-1}x\|_2^2$$ therefore the function $e^{-\pi\langle M^{-1}x,x\rangle}$ is log-concave. Then, applying \eqref{revers_log_conc} yields
\begin{equation*}
    \int_{\mathbb{R}^d}e^{-\pi d\langle M^{-1}x,x\rangle}dx\ge\frac{1}{BL_2}\frac{\prod_{j=1}^md_j^{p_jd_j}}{d^d}\prod_{j=1}^m\Bigg(\int_{F_j}e^{-\pi d_j\langle M^{-1}x_j,x_j\rangle}dx_j\Bigg)^{p_j}
\end{equation*}

We make the change of variables $y=\sqrt{d}x$ and $y_j=\sqrt{d_j}x_j$ to the left-hand integral and the right-hand integrals respectively
\begin{equation*}
    \int_{\mathbb{R}^d}e^{-\pi\langle M^{-1}y,y\rangle}dy\ge\frac{1}{BL_2}\frac{\prod_{j=1}^md_j^{p_jd_j/2}}{d^{d/2}}\prod_{j=1}^m\Bigg(\int_{F_j}e^{-\pi\langle M^{-1}x_j,x_j\rangle}dx_j\Bigg)^{p_j},
\end{equation*}
which translates into
\begin{equation}\label{det_ineq_2}
    (\det M)^{1/2}\ge\frac{1}{BL_2}\frac{\prod_{j=1}^md_j^{p_jd_j/2}}{d^{d/2}}\prod_{j=1}^m\det(P_{F_j}MP_{F_j}^T)^{p_j/2}.
\end{equation}
Finally \eqref{vol-proj} implies 
\begin{equation*}
    |\mathcal E|\ge\frac{1}{BL_2}\frac{\omega_d}{d^{d/2}}\prod_{j=1}^{m}\Bigg(\frac{d_j^{d_j/2}}{\omega_{d_j}}\Bigg)^{p_j}\prod_{j=1}^{m}|P_{F_j}(\mathcal E)|^{p_j}.
\end{equation*}
\end{proof}
To express the aforementioned inequality in terms of the $\sigma$ parameters, we begin by representing the ellipsoid as a projection body. Consider an ellipsoid $\mathcal{E}=A(B_2^d)$. Apply the following properties of the projection body $\Pi$:
$$\Pi R(L)=|\det(R)|R^{-T}\Pi L,\ \text{for any}\ R\in GL_d,\quad \text{and}\quad \Pi (rL)=r^{d-1}\Pi L,\ \text{for any}\ r>0,$$
to obtain
\begin{equation}\label{ell-proj}
\mathcal{E}=\Pi\big(R(S^{d-1})\big),
\end{equation}
where $R$ is the linear transformation given by
\begin{equation}\label{lin-tran}
R=\Bigg(\frac{|\det(A)|}{2\omega_{d-1}}\Bigg)^{\frac{1}{d-1}}A^{-T}.
\end{equation}
A combination of the above yiels the following Theorem.
\begin{theorem}
    Let $\mathbb{S}=(S,v,\sigma)$ a $d$-hypersurface and $\{w_1,\dots,w_d\}$ a basis of $\mathbb{R}^d$. Let $m\ge1$ and let $(\sigma_1,\dots,\sigma_m)$ a uniform cover of $[d]$ with weights $(p_1,\dots,p_m)$. Let $F_j=span\{w_k:k\in\sigma_j\}$ and $dim(F_j)=|\sigma_j|=d_j$. We can find a $d$-hypersurface $\mathbb{G}$ related to $\mathbb{S}$ such that,
    \begin{equation}\label{doub-ineq}
        \frac{c_1}{BL_2^{1/d}}\prod_{j=1}^m\sigma(F_j,\mathbb{G})^{p_j/d}\le \mathrm{vis}^2(\mathbb{S})\le c_2BL_2^{1/d}\prod_{j=1}^m\sigma(F_j,\mathbb{G})^{p_j/d},
    \end{equation}
where
    \begin{equation}
        c_1=\frac{1}{\omega_d^{1/d}\sqrt{d}}\prod_{j=1}^m\bigg(\frac{2^{d_j}d_j^{d_j/2}}{\omega_{d_j}d_j!}\bigg)^{p_j/d} \ \ \text{and} \ \ \ \ c_2=\frac{1}{\omega_d^{1/d}}\prod_{j=1}^m\bigg(\frac{2^{d_j}}{\omega_{d_j}d_j!}\bigg)^{p_j/d}.
    \end{equation}
    \end{theorem}
    \begin{proof}
Since $K^2(\mathbb{S})$ is an ellipsoid     \begin{equation}\label{vis_ellips_1}\mathrm{vis}^2(\mathbb{S})^d=\frac{1}{|K^2(\mathbb{S})|}=\frac{1}{\omega_d^2}|K^2(\mathbb{S})^{\circ}|.
    \end{equation}
Moreover, there is $T\in GL_d$ such that $K^2(\mathbb{S})=T(B_2^d)$. Define $R\in GL_d$ as in \eqref{lin-tran} to write $K^2(\mathbb{S})^{\circ}=\Pi\big(R(S^{d-1})\big)$.
Consider the vector field $v(x)=x$ and the  measure $\nu$ by $$d\nu=\frac{d|K^2(\mathbb{S})|}{2\omega_{d-1}}h_{K^2(\mathbb{S})}^{-(d+1)}d\sigma_{S^{d-1}}.$$ It is classical that this is the generating measure of $K^2(\mathbb{S})^{\circ}$ as a zonoid where $\sigma_{S^{d-1}}$ is the normalized cone measure on $S^{d-1}$. Define the $d$-hypersurface $\mathbb{G}=(S^{d-1},id,\nu)$.
From this point of view it is immediate that  
    \begin{equation}
        K^2(\mathbb{S})^{\circ}=\Pi(\mathbb{G}) \ \ \text{and} \ \ |P_E(\Pi(\mathbb{G}))|=\frac{2^k}{k!}\sigma(E,\mathbb{G}), \ \ E\in G_{d,k}   
    \end{equation}
and \eqref{doub-ineq} follows from  $\eqref{main-ell}$ for $\Pi(\mathbb{G})$.
    \end{proof}

\begin{remark}
For the sake of completeness we compute the measure $\nu$. Let $\mathcal{E}$ be an ellipsoid. We wish to find  the generating measure of
$\mathcal{E}$, as a zonoid. We calculate the integral    $$\int_{\mathcal{E}}|\langle x,y\rangle|dy$$
in two ways.
First, we make a change to polar coordinates
\begin{align*}
    &\int_{\mathcal{E}}|\langle x,y\rangle|dy=d\omega_d\int_{S^{d-1}}\int_0^{1/\|\theta\|_{\mathcal{E}}}r^{d-1}|\langle x,r\theta\rangle|drd\sigma_{S^{d-1}}(\theta)\\
    &=d\omega_d\int_{S^{d-1}}|\langle x,\theta\rangle|\bigg(\int_0^{1/\|\theta\|_{\mathcal{E}}}r^ddr\bigg)d\sigma_{S^{d-1}}(\theta)=\frac{d\omega_d}{d+1}\int_{S^{d-1}}|\langle x,\theta\rangle|\|\theta\|_{\mathcal{E}}^{-(d+1)}\sigma_{S^{d-1}}(\theta).
\end{align*}
Next, if $E=A(B_2^d)$ with $A\in GL_d$, then  $|\det A|=\frac{|\mathcal{E}|}{\omega_d}$ and
$h_{\mathcal{E}}(x)=\|A^Tx\|_2$. 
Using the rotational invariance 
\begin{align*}
\int_{\mathcal{E}}|\langle x,y\rangle|dy=|\det A|\int_{B_2^d}|\langle x,Ay\rangle|dy
    &=|\det A|\int_{B_2^d}|\langle A^Tx,y\rangle|dy \\
    &=\frac{2\omega_{d-1}}{d+1}|\det A|\|A^Tx\|_2\\&=\frac{2\omega_{d-1}}{d+1}\frac{|\mathcal{E}|}{\omega_d}h_{\mathcal{E}}(x).
\end{align*}
Combining the two calculations, we conclude that
$$h_{\mathcal{E}}(x)=\int_{S^{d-1}}|\langle x,\theta\rangle|d\nu(\theta)$$
with
$$d\nu(\theta)=\frac{d\omega_d^2}{2\omega_{d-1}|\mathcal{E}|}d\sigma_{S^{d-1}}(\theta).$$
\end{remark}
\bigskip

\noindent In what follows, we derive lower and upper bounds for $\mathrm{vis}^2(\mathbb{S})$ by employing the corresponding $\sigma$ quantities.

For $E\in G_{d,k}$ and the $d-$hypersurface $\mathbb{S}=(S,\sigma,v)$ define
\begin{align*}
 \sigma_2(E,\mathbb{S}):&=\bigg(\int_S|E^{\perp}\wedge v(x_1)\wedge\dots\wedge v_k|^2d\sigma(x_1)\dots d\sigma(x_k)\bigg)^{1/2}\\
    &=\bigg(\int_S|P_E(v(x_1))\wedge\dots\wedge P_E(v_k)|^2d\sigma(x_1)\dots d\sigma(x_k)\bigg)^{1/2}.
\end{align*}
For the covariance matrix $T_{\mathbb{S}}$ of $\mathbb{S}$ we have
$$T_{\mathbb{S}}=\int_Sv(x)\otimes v(x)d\sigma(x)$$ and we know that
\begin{equation}\label{T_S_1}
    d!\det(T_{\mathbb{S}})=\int_S\dots\int_S|v(x_1)\wedge\dots\wedge v(x_d)|^2d\sigma(x_1)\dots d\sigma(x_d)=Q_d^2(\mathbb{S)}^{2d}
\end{equation}
Since
$$P_ET_{\mathbb{S}}P_E=\int_SP_E(v(x))\otimes P_E(v(x))d\sigma(x)$$
using \eqref{T_S_1} implies
\begin{equation}
\sigma_2(E,\mathbb{S})=\sqrt{k!\det(P_ET_{\mathbb{S}}P_E)}.
\end{equation}
\begin{proposition}\label{vis_2_prop}
     Let $\mathbb{S}=(S,\sigma,v)$ a generalised d-hypersurface and$\{w_1,\dots,w_d\}$ a basis of $\mathbb{R}^d$, let $m\ge1$ and let $(\sigma_1,\dots,\sigma_m)$ a uniform cover of $[d]$ with weights $(p_1,\dots,p_m)$. Let $E_j=span\{w_k:k\in\sigma_j\}$ and $dim(E_j)=|\sigma_j|=d_j$
     
     \begin{equation*}
         \sqrt{\frac{d!}{d^d}}\prod_{i=1}^m\bigg(\frac{d_i^{d_i}}{d_i!}\bigg)^{p_i/2}\frac{1}{BL_2}\prod_{i=1}^m\sigma_2(E_i,\mathbb{S})^{p_i}\le Q_d^2(\mathbb{S})^d\le\sqrt{d!}\prod_{i=1}^m\bigg(\frac{1}{d_i!}\bigg)^{p_i/2}BL_2\prod_{i=1}^m\sigma_2(E_i,\mathbb{S})^{p_i}
     \end{equation*}
\end{proposition}
\begin{proof}
Using the above representation for $\sigma_2$, Proposition follows from \eqref{det-ineq} and \eqref{det_ineq_2}. 
\end{proof}
\begin{theorem}
  Let $\mathbb{S}=(S,\sigma,v)$ be a generalised $d$-hypersurface. Then,
    \begin{equation}
 \mathrm{vis}^2(\mathbb{S})=\bigg(\frac{1}{\omega_d}\prod_{i=1}^d\sigma_2(e_i,\mathbb{S})\bigg)^{1/d}
    \end{equation}
    where $\{e_1,\dots,e_d\}$ is the usual orthonormal basis of $\mathbb{R}^d$. Moreover,
    \begin{equation}
        \frac{1}{(\sqrt{d!\omega_d})^{1/d}}Q_d^2(\mathbb{S})\le \mathrm{vis}^2(\mathbb{S})\le\frac{\sqrt{d}}{(\sqrt{d!\omega_d})^{1/d}}Q_d^2(\mathbb{S})
    \end{equation}
\end{theorem}
\medskip
\begin{proof}
Since $K^2(\mathbb{S})$ is an ellipsoid     \begin{equation}\label{vis_ellips_1}\mathrm{vis}^2(\mathbb{S})^d=\frac{1}{|K^2(\mathbb{S})|}=\frac{1}{\omega_d^2}|K^2(\mathbb{S})^{\circ}|.
    \end{equation}
Moreover, the length of each axis is
    $$h_{K^2(\mathbb{S})^o}(e_i)=\|e_i\|_{K^2(\mathbb{S})}=\bigg(\int_S|\langle e_i,v(x)\rangle|^2d\sigma(x)\bigg)^{1/2}=\bigg(\int_S| e_i^{\perp}\wedge v(x)|^2d\sigma(x)\bigg)^{1/2}=\sigma_2(e_i,\mathbb{S}).$$
 Therefore,  $$|K^2(\mathbb{S})^{\circ}|=\omega_d\prod_{i=1}^d\sigma_2(e_i,\mathbb{S}).$$
    Substituting into \eqref{vis_ellips_1}, we obtain the first assertion.
    For the inequality, apply Proposition \ref{vis_2_prop} for $E_i=\mathrm{span}(e_i)$ and $p_i=1$. Since $BL_2=1$, this yields
    $$\sqrt{\frac{d!}{d^d}}\prod_{i=1}^d\sigma_2(e_i,\mathbb{S})\le Q_d^2(\mathbb{S})\le\sqrt{d!}\prod_{i=1}^d\sigma_2(e_i,\mathbb{S})$$
    Substituting into the above double inequality, we obtain the second assertion. 
\end{proof}
\subsection{Bounds for all  \texorpdfstring{$p\geq 1$}{p >= 1}.}
To start with we offer an upper bound for $\mathrm{vis}^p(\mathbb{S})$.
\begin{proposition}\label{vis_p}
    Let $\mathbb{S}=(S,\sigma,v)$ be a hypersurface. For every $p\ge1$ we have 
    \begin{equation*}      \mathrm{vis}^p(\mathbb{S})=\frac{1}{|K^p(\mathbb{S})|^{1/d}}\le\frac{c_{d,p}^{1/p}}{\omega_d^{1/d}}Q_1^p(\mathbb{S}),
    \end{equation*}
where $c_{d,p}=\frac{d+p}{p}\frac{\omega_{d-1}}{\omega_d}\frac{\Gamma\big(\frac{p+1}{2}\big)\Gamma\big(\frac{d+1}{2}\big)}{\Gamma\big(\frac{d+p+1}{2}\big)}$.
\end{proposition}

\begin{proof}
    Our starting point will be the following inequality (see for example in \cite[p.76]{Milman1989}) that holds for every convex body $K$ and every symmetric convex body $C$ and $p>0$.
    \begin{equation}
        \Bigg(\frac{|K|}{|C|}\Bigg)^{1/d}\le\Bigg(\frac{d+p}{d}\frac{1}{|K|}\int_{K}\|x\|_{C}^pdx\Bigg)^{1/p}
    \end{equation}
    For $K=B_2^d$ and $C=K^p(\mathbb{S})$ the above inequality becomes
    \begin{equation*}
         \Bigg(\frac{\omega_d}{|K^p(\mathbb{S})|}\Bigg)^{1/d}\le\Bigg(\frac{d+p}{d\omega_d}\int_{B_2^d}\|x\|_{K^p(\mathbb{S})}^pdx\Bigg)^{1/p}
    \end{equation*}
    Using polar coordinates to evaluate we get
    \begin{align*}
        &\int_{B_2^d}\|x\|_{K^p(\mathbb{S})}^pdx
        =d\omega_d\int_{S^{d-1}}\int_0^1\|r\theta\|_{K^p(\mathbb{S})}^pr^{d-1}drd\sigma_{S^{d-1}}(\theta)\\
        &=d\omega_d\Bigg(\int_0^1r^{d+p-1}dr\Bigg)\Bigg(\int_{S^{d-1}}\|\theta\|_{K^p(\mathbb{S})}^pd\sigma_{S^{d-1}}(\theta)\Bigg)\\
        &=\frac{d\omega_d}{d+p}\int_{S^{d-1}}\|\theta\|_{K^p(\mathbb{S})}^pd\sigma_{S^{d-1}}(\theta).
    \end{align*}
Using the definition of the norm of the convex body $K^p(\mathbb{S})$ we obtain $\|\theta\|_{K^p(\mathbb{S})}^p=\int_{S}|\langle\theta,v(x)\rangle|^pd\sigma(x)$ and the formula $\int_{S^{d-1}}|\langle\theta,y\rangle|^pd\sigma_{S^{d-1}}(\theta)=c_{d,p}\|y\|_2^p$ for $y=v(x)$ implies
    \smallskip
    \begin{align*}
        &\int_{S^{d-1}}\|\theta\|_{K^p(\mathbb{S})}^pd\sigma_{S^{d-1}}(\theta)
        =\int_{S^{d-1}}\int_{S}|\langle\theta,v(x)\rangle|^pd\sigma(x)d\sigma_{S^{d-1}}(\theta)\\
        &=\int_{S}\int_{S^{d-1}}|\langle\theta,v(x)\rangle|^pd\sigma_{S^{d-1}}(x)d\sigma_{S}(\theta)\\
        &=c_{d,p}\int_{S}\|v(x)\|_2^pd\sigma(x)\\
        &=c_{d,p}Q_1^p(\mathbb{S})^p.
    \end{align*}
    Combining all the above, we obtain the upper bound. 
\end{proof}

\noindent The transition from $Q_1^p$ to $Q_d^p$ is achieved by making use of Lewis' position.
\begin{proposition}\label{thm:inf-a-vs-I}
With the notation above we have
\[
Q_{d}^p(\mathbb{S}) \;\le\; \inf_{\substack{A\in GL(d)\\|\det A|=1}} a(A) \;\le\; \Big(\frac{d^d}{d!}\Big)^{1/(2d)}\; Q_{d}^p(\mathbb{S}),
\]
where \[
a(A):=\Big(\int_S\|A v(x)\|_2^p\,d\sigma(x)\Big)^{1/p}.
\]
\end{proposition}

\begin{proof}
First we prove the lower bound. Fix $A\in GL(d)$ with $|\det A|=1$. For every $(x_1,\dots,x_d)\in S^d$ Hadamard's inequality gives
\[
\big|A v(x_1)\wedge\dots\wedge A v(x_d)\big|
\le \prod_{i=1}^d \|A v(x_i)\|_2.
\]
Therefore, since $|\det(A)|=1$,
\[
\int_{S^d} \big| v(x_1)\wedge\dots\wedge  v(x_d)\big|^p\,d\sigma^{\otimes d}
\le \int_{S^d}\prod_{i=1}^d \|A v(x_i)\|_2^p
= \Big(\int_S \|A v(x)\|_2^p\,d\sigma(x)\Big)^d.
\]
This is true for every $A$ with $|\det A|=1$, therefore we obtain the left-hand side  inequality.
\bigskip

We turn to the upper bound. Let $u$ be the matrix that satisfies the equalities from Lemma~\ref{lewis_lemma}, $z(x):=u v(x)$, $w(x):=\|u v(x)\|_2^{p-2}$ and the matrix 
\[
M:=\int_S w(x)\,z(x)\otimes z(x)\,d\sigma(x)
\]
equals $\tfrac{1}{d}I_d$, so $\det M=(1/d)^d$. Thus, from the properties of mixed discriminants we obtain 
\[
\Big(\frac{1}{d}\Big)^d
= \frac{|\det u|^2}{d!}\int_{S^d} \big| v(x_1)\wedge\dots\wedge  v(x_d)\big|^2 \prod_{i=1}^d \|u v(x_i)\|_2^{p-2}\,d\sigma^{\otimes d}.
\]
Applying Hölder's inequality on $S^d$ with exponents $\alpha=\tfrac{p}{2}$ and $\beta=\tfrac{p}{p-2}$ we get
\[
\begin{aligned}
&\int_{S^d} \big| v(x_1)\wedge\dots\wedge  v(x_d)\big|^2 \prod_{i=1}^d \|u v(x_i)\|_2^{p-2}\,d\sigma^{\otimes d}
\le
Q_j^p(\mathbb{S})^{2d}
\Big(\int_{S^d}\prod_{i=1}^d \|u v(x_i)\|_2^{p}\,d\sigma^{\otimes d}\Big)^{(p-2)/p}.
\end{aligned}
\]
The second factor factorizes:
\[
\int_{S^d}\prod_{i=1}^d \|u v(x_i)\|_2^{p}\,d\sigma^{\otimes d}
= \Big(\int_S \|u v(x)\|_2^{p}\,d\sigma(x)\Big)^d = a(u)^{pd} = 1^{pd} = 1,
\]
because $a(u)=1$ by Lewis' normalization. Hence
\[
\int_{S^d} \big| v(x_1)\wedge\dots\wedge  v(x_d)\big|^2 \prod_{i=1}^d \|u v(x_i)\|_2^{p-2}\,d\sigma^{\otimes d}
\le Q_j^p(\mathbb{S})^{2d}.
\]
Plugging this bound into the previous identity yields
\[
\Big(\frac{1}{d}\Big)^d \le \frac{|\det u|^2}{d!}\, Q_j^p(\mathbb{S})^{2d},
\]
hence
\begin{equation}\label{det_u_ineq}
|\det u| \ge \Big(\frac{d!}{d^d}\Big)^{1/2}\, Q_j^p(\mathbb{S})^{-d}.
\end{equation}
Set
\[
A := (\det u)^{-1/d}\, u.
\]
Then $|\det A|=1$ and using homogeneity of $a(\cdot)$ we have \(
a(A) = (\det u)^{-1/d}\, a(u) = (\det u)^{-1/d}.\) From the lower bound on $|\det u|$ we obtain
\[
a(A) \le \Big(\frac{d^d}{d!}\Big)^{1/(2d)}\, Q_j^p(\mathbb{S}),
\]
which completes the proof. 
\end{proof}
\noindent The above proposition allows us to prove our main Theorem for this section, completing the comparison of $\mathrm{vis}^p$ and $Q_j^p$.

\begin{theorem}\label{thm:vis-vs-I}
With the notation above there exist positive constants $C_1(d,p),C_2(d,p)$ depending only on $d,p$ such that
\[
C_1(d,p)\; Q_d^p(\mathbb{S}) \;\le\; \mathrm{vis}^p(\mathbb S) \;\le\; C_2(d,p)\; Q_d^p(\mathbb{S}).
\]
Combining with Proposition~\ref{thm:inf-a-vs-I}, one has
\[
\mathrm{vis}^p(\mathbb S) \;\asymp_{d,p}\; \inf_{\substack{A\in GL(d)\\|\det A|=1}} a(A)
\;\asymp_{d,p}\; Q_d^p(\mathbb{S}).
\]
\end{theorem}

\begin{proof}
We first derive the upper bound for $\mathrm{vis}^p$.
Using the Proposition \ref{thm:inf-a-vs-I} there exists \(A_0\in GL(d)\) with \(|\det A_0|=1\) such that
\[
\left(\int_S \|A_0 v(x)\|_2^p\,d\sigma(x)\right)^{1/p} \le C_{d}\,Q_d^p(\mathbb{S}).
\]
Observe that if we replace the field \(v\) by \(v^{(A_0)}:=A_0 v\), then the associated body \(K^p\) transforms by the linear map \((A_0^T)^{-1}\) on the ambient space and hence
\[
|K^p(\mathbb S^{(A_0)})| = |\det(A_0^T)^{-1}|\, |K^p(\mathbb S)| = |\det A_0|^{-1}\,|K^p(\mathbb S)|.
\]
Since \(|\det A_0|=1\), the visibility is invariant:
\[
\mathrm{vis}^p(\mathbb S^{(A_0)})=\mathrm{vis}^p(\mathbb S).
\]
Using the bound from Proposition \ref{vis_p} to the vector field \(v^{(A_0)}=A_0 v\) we obtain
\[
\mathrm{vis}^p(\mathbb S)
= \mathrm{vis}^p(\mathbb S^{(A_0)})
\le c_{d,p}\left(\int_S \|A_0 v(x)\|_2^p\,d\sigma(x)\right)^{1/p}
\le c_{d,p}\, C_{d}\, Q_d^p(\mathbb{S}).
\]

For the other bound, by Lewis' lemma as previously, there exists $u\in GL(d)$ with $a(u)=1$, the dual normalization $a^*(u^{-1})=d$ and the the isotropic identity
\begin{equation}\label{iso}
\int_S \|u v(x)\|_2^{p-2}\,(u v(x))(u v(x))^T\,d\sigma(x)=\frac{1}{d}\,I_d.
\end{equation}
Set $z(x):=u v(x)$ and define the measure $\mu$ on $S$ by $d\mu(x):=\|z(x)\|_2^p d\sigma(x)$. Taking trace in \eqref{iso} gives $$\mu(S)=d\int_S\|z\|_2^p d\sigma =d.$$

We first use the well-known formula for the volume of the polar body
\[
\Gamma\!\Big(\frac{d}{p}+1\Big)\,|K^p(z)| \;=\; 
\int_{\mathbb R^d}\exp\!\Big(-\frac{1}{d}\int_S |\langle y,\tfrac{z(x)}{\|z(x)\|_2}\rangle|^p \,d\mu(x)\Big)\,dy.
\]
Since $\int_S \big(\tfrac{z}{\|z\|_2}\big)\otimes\big(\tfrac{z}{\|z\|_2}\big)\,d\mu = I_d$, we are in position to apply the continuous Brascamp–Lieb inequality, Lemma \ref{cont_brascamp_lieb} for the isotropic measure $\mu$ and the one–dimensional functions $G_x(t)=\exp\!\big(-\tfrac{1}{d}|t|^p\big)$, which yields 
\[
\Gamma\!\Big(\frac{d}{p}+1\Big)\,|K^p(z)|
\le \exp\!\Big(\int_S \log\Big(\int_{\mathbb R} e^{-\frac{1}{d}|t|^p}\,dt\Big)\,d\mu(x)\Big)
= \Big(\int_{\mathbb R} e^{-\frac{1}{d}|t|^p}\,dt\Big)^d.
\]
The one–dimensional integral rescales as
\[
\int_{\mathbb R} e^{-\frac{1}{d}|t|^p}\,dt = d^{1/p}\int_{\mathbb R} e^{-|s|^p}\,ds = d^{1/p}\,2\Gamma\!\Big(\frac{1}{p}+1\Big).
\]
Thus
\[
|K^p(z)| \le \frac{d^{d/p}\,(2\Gamma(\tfrac{1}{p}+1))^d}{\Gamma(\tfrac{d}{p}+1)}.
\]
Therefore
\begin{equation}\label{vis-z-lower}
\mathrm{vis}^p(\mathbb S_z)=|K^p(z)|^{-1/d} \ge
\frac{\Gamma(\tfrac{d}{p}+1)^{1/d}}{d^{1/p}\,2\Gamma(\tfrac{1}{p}+1)}=:c_0(d,p).
\end{equation}

To transfer this bound from $z=u v$ back to the original field $v$ we use the linear equivariance of $K^p$: one has $K^p(z)=u^{-T}K^p(v)$, hence $|K^p(z)|=|\det u|^{-1}|K^p(v)|$. Consequently
\[
\mathrm{vis}^p(\mathbb S)=|K^p(v)|^{-1/d}=|\det u|^{-1/d}\,\mathrm{vis}^p(\mathbb S_z).
\]
Combining this with \eqref{vis-z-lower} yields
\[
\mathrm{vis}^p(\mathbb S) \ge c_0(d,p)\;|\det u|^{-1/d}.
\]
It remains to lower bound $|\det u|^{-1/d}$ in terms of $Q_j^p(\mathbb{S})$. Observe that
\[
\int_{S^d}\big|v(x_1)\wedge \dots\wedge v(x_d)\big|^p
=|\det u|^{-p}\int_{S^d}\big|z(x_1)\wedge\dots\wedge z(x_d)\big|^p.
\]
By Hadamard's inequality, for each $(x_1,\dots,x_d)\in S^d$,
\[
\big|z(x_1)\wedge\dots\wedge z(x_d)\big|
\le \prod_{i=1}^d \|z(x_i)\|_2.
\]
Raising to the power \(p\) and integrating over \(S^d\) gives
\[
\int_{S^d}\big|\det[z(\mathbf x)]\big|^p
\le \int_{S^d}\prod_{i=1}^d \|z(x_i)\|_2^p
= \Big(\int_S \|z(x)\|_2^p d\sigma(x)\Big)^d = a(u)^{pd}=1.
\]
Therefore,
\[
Q_d^p(\mathbb{S})^{d} \le |\det u|^{-1}.
\]
Combining with the previous displayed inequality yields
\[
\mathrm{vis}^p(\mathbb S) \ge c_0(d,p)\; Q_d^p(\mathbb{S}).
\]
\end{proof}

\textbf{Acknowledgements.} The  authors acknowledge support by the Hellenic Foundation for Research and Innovation (H.F.R.I.) under the call “Basic research Financing (Horizontal support of all Sciences)” under the National Recovery and Resilience Plan “Greece 2.0” funded by the European Union–NextGeneration EU(H.F.R.I. Project Number:15445).

\bibliographystyle{plain}

\bibliography{geometric_quantities}

@article{Reisnzer-Gordon-Meyer,
 ISSN = {00029939, 10886826},
 URL = {http://www.jstor.org/stable/2047501},
 author = {Y. Gordon and M. Meyer and S. Reisner},
 journal = {Proceedings of the American Mathematical Society},
 number = {1},
 pages = {273--276},
 publisher = {American Mathematical Society},
 title = {Zonoids with Minimal Volume-Product--A New Proof},
 urldate = {2025-10-11},
 volume = {104},
 year = {1988}
}

@article{BourgainMilman1987,
  author    = {J. Bourgain and V. D. Milman},
  title     = {New volume ratio properties for convex symmetric bodies in $\mathbb{R}^n$},
  journal   = {Inventiones Mathematicae},
  year      = {1987},
  volume    = {88},
  number    = {2},
  pages     = {319--340},
  doi       = {10.1007/BF01388911},
  url       = {https://doi.org/10.1007/BF01388911},
  issn      = {1432-1297}
}

@article{BrazitikosCarberyMacIntyre2023,
  author  = {Brazitikos, Silouanos and Carbery, Anthony and MacIntyre, Finlay},
  title   = {On some integral-geometric quantities related to transversality, curvature and visibility},
  journal = {Indiana Univ. Math. J.},
  note    = {To appear},
  url     = {https://arxiv.org/pdf/2310.12789},
  eprint  = {2310.12789},
  archivePrefix = {arXiv}
}

@article {Guth,
    AUTHOR = {Guth, Larry},
     TITLE = {The endpoint case of the {B}ennett-{C}arbery-{T}ao multilinear
              {K}akeya conjecture},
   JOURNAL = {Acta Math.},
  FJOURNAL = {Acta Mathematica},
    VOLUME = {205},
      YEAR = {2010},
    NUMBER = {2},
     PAGES = {263--286},
      ISSN = {0001-5962},
   MRCLASS = {42B15 (28A75 52A38)},
  MRNUMBER = {2746348},
MRREVIEWER = {Dmitry Ryabogin},
       DOI = {10.1007/s11511-010-0055-6},
       URL = {https://doi.org/10.1007/s11511-010-0055-6},
}

@Inbook{Barthe2004,
author="Barthe, F.",
title="A Continuous Version of the Brascamp--Lieb Inequalities",
bookTitle="Geometric Aspects of Functional Analysis: Israel Seminar 2002-2003",
year="2004",
publisher="Springer Berlin Heidelberg",
address="Berlin, Heidelberg",
pages="53--63",
isbn="978-3-540-44489-3",
doi="10.1007/978-3-540-44489-3_6",
url="https://doi.org/10.1007/978-3-540-44489-3_6"
}

@Inbook{Milman1989,
author="Milman, V. D.
and Pajor, A.",
editor="Lindenstrauss, Joram
and Milman, Vitali D.",
title="Isotropic position and inertia ellipsoids and zonoids of the unit ball of a normed n-dimensional space",
bookTitle="Geometric Aspects of Functional Analysis: Israel Seminar (GAFA) 1987--88",
year="1989",
publisher="Springer Berlin Heidelberg",
address="Berlin, Heidelberg",
pages="64--104",
isbn="978-3-540-46189-0",
doi="10.1007/BFb0090049",
url="https://doi.org/10.1007/BFb0090049"
}

@article{KoldobskySaroglouZvavitch2019,
  author    = {Alexander Koldobsky and Christos Saroglou and Artem Zvavitch},
  title     = {Estimating volume and surface area of a convex body via its projections or sections},
  journal   = {Studia Mathematica},
  year      = {2019},
  volume    = {244},
  pages     = {245--264},
  doi       = {10.4064/sm8774-10-2017},
  note      = {Published online: 29 June 2018},
  url       = {https://www.impan.pl/en/publishing-house/journals-and-series/studia-mathematica/online/112534/estimating-volume-and-surface-area-of-a-convex-body-via-its-projections-or-sections}
}

@article{Campi2018,
  author    = {Stefano Campi and Peter Gritzmann and Paolo Gronchi},
  title     = {On the Reverse {L}oomis--{W}hitney Inequality},
  journal   = {Discrete \& Computational Geometry},
  year      = {2018},
  volume    = {60},
  number    = {1},
  pages     = {115--144},
  doi       = {10.1007/s00454-017-9949-9},
  url       = {https://doi.org/10.1007/s00454-017-9949-9},
  issn      = {1432-0444}
}

@article{Brazitikos-Giannopoulos_BL,
  author    = {Silouanos Brazitikos and Apostolos Giannopoulos},
  title     = {Continuous Version of the Approximate Geometric {B}rascamp--{L}ieb Inequalities},
  journal   = {The Journal of Geometric Analysis},
  year      = {2022},
  volume    = {32},
  number    = {6},
  pages     = {174},
  doi       = {10.1007/s12220-022-00909-z},
  url       = {https://doi.org/10.1007/s12220-022-00909-z},
  issn      = {1559-002X}
}

@article{BrazitikosGiannopoulosLiakopoulos+2018+345+354,
url = {https://doi.org/10.1515/advgeom-2017-0063},
title = {Uniform cover inequalities for the volume of coordinate sections and projections of convex bodies},
title = {},
author = {Silouanos Brazitikos and Apostolos Giannopoulos and Dimitris-Marios Liakopoulos},
pages = {345--354},
volume = {18},
number = {3},
journal = {Advances in Geometry},
doi = {doi:10.1515/advgeom-2017-0063},
year = {2018},
lastchecked = {2025-10-06}
}

@article{Tilli,
 ISSN = {00029947},
 URL = {http://www.jstor.org/stable/25733379},
 author = {P. Tilli},
 journal = {Transactions of the American Mathematical Society},
 number = {9},
 pages = {4497--4509},
 publisher = {American Mathematical Society},
 title = {ISOPERIMETRIC INEQUALITIES FOR CONVEX HULLS AND RELATED QUESTIONS},
 urldate = {2025-10-06},
 volume = {362},
 year = {2010}
}

@article{alonso2021reverse,
  title={Reverse {L}oomis--{W}hitney inequalities via isotropicity},
  author={Alonso-Guti{\'e}rrez, David and Brazitikos, Silouanos},
  journal={Proceedings of the American Mathematical Society},
  volume={149},
  number={2},
  pages={817--828},
  year={2021}
}

@book {GianMilArt,
    AUTHOR = {Artstein-Avidan, Shiri and Giannopoulos, Apostolos and Milman,
              Vitali D.},
     TITLE = {Asymptotic geometric analysis. {P}art {I}},
    SERIES = {Mathematical Surveys and Monographs},
    VOLUME = {202},
 PUBLISHER = {American Mathematical Society, Providence, RI},
      YEAR = {2015},
     PAGES = {xx+451},
      ISBN = {978-1-4704-2193-9},
   MRCLASS = {52A21 (28Axx 46-02 46Bxx 52A23 52A40)},
  MRNUMBER = {3331351},
MRREVIEWER = {Artem Zvavitch},
       DOI = {10.1090/surv/202},
       URL = {https://doi.org/10.1090/surv/202},
}

@article {Zhang,
    AUTHOR = {Zhang, Ruixiang},
     TITLE = {The endpoint perturbed {B}rascamp--{L}ieb inequalities with
              examples},
   JOURNAL = {Anal. PDE},
  FJOURNAL = {Analysis \& PDE},
    VOLUME = {11},
      YEAR = {2018},
    NUMBER = {3},
     PAGES = {555--581},
      ISSN = {2157-5045},
   MRCLASS = {42B10 (26D15)},
  MRNUMBER = {3738255},
MRREVIEWER = {Javier Duoandikoetxea},
       DOI = {10.2140/apde.2018.11.555},
       URL = {https://doi.org/10.2140/apde.2018.11.555},
}

@article {Z-K,
    AUTHOR = {Zorin-Kranich, Pavel},
     TITLE = {Kakeya--{B}rascamp--{L}ieb inequalities},
   JOURNAL = {Collect. Math.},
  FJOURNAL = {Collectanea Mathematica},
    VOLUME = {71},
      YEAR = {2020},
    NUMBER = {3},
     PAGES = {471--492},
      ISSN = {0010-0757},
   MRCLASS = {26D15 (42B99 52C07)},
  MRNUMBER = {4129538},
       DOI = {10.1007/s13348-019-00273-2},
       URL = {https://doi.org/10.1007/s13348-019-00273-2},
}

@article{GBBC,
author = {Alonso-Gutiérrez, David and Bernués, Julio and Brazitikos, Silouanos and Carbery, Anthony},
title = {On affine invariant and local {L}oomis--{W}hitney type inequalities},
journal = {Journal of the London Mathematical Society},
volume = {103},
number = {4},
pages = {1377-1401},
doi = {https://doi.org/10.1112/jlms.12411},
url = {https://londmathsoc.onlinelibrary.wiley.com/doi/abs/10.1112/jlms.12411},
eprint = {https://londmathsoc.onlinelibrary.wiley.com/doi/pdf/10.1112/jlms.12411},
year = {2021}
}
\section{Appendix}
\begin{lemma}[Continuous Brascamp--Lieb inequality on $S$]
Let $S$ be a hypersurface equipped with a positive measure $\mu$.
Let $v:S\to\mathbb{R}^n$ be measurable with $v(x)\neq0$ for $\mu$‑almost every $x\in S$, and assume the normalization
\begin{equation}\label{eq:identity_assumption}
I_d= \int_S \frac{v(x)\otimes v(x)}{\|v(x)\|_2^2}\,d\mu(x).
\end{equation}
Define the unit directions $u(x):=v(x)/\|v(x)\|_2\in S^{n-1}$ and suppose we are given a measurable family
of non-negative functions ${f_x:\mathbb R\to[0,\infty)},{x\in S}$ that meets the usual technical integrability
conditions.

Then the following inequality holds:
\begin{equation}\label{eq:conclusion_S}
\int_{\mathbb R^n}\exp\Big(\int_S\log\bigl(f_x(\langle y,u(x)\rangle)\bigr)\,d\mu(x)\Big)\,dy
\le
\exp\Big(\int_S\log\Big(\int_{\mathbb R} f_x(t)\,dt\Big),d\mu(x)\Big).
\end{equation}
\end{lemma}

\begin{proof}
The proof uses the push‑forward of $\mu$ under the map $u:S\to S^{n-1}$ and the disintegration of $\mu$ with respect to that map.

Define the push‑forward measure $\nu:=u_{\#}\mu$ on $S^{n-1}$ by the rule that 
for every measurable $\Phi:S^{n-1}\to\mathbb R$,

$$
\int_S\Phi(u(x))\,d\mu(x)=\int_{S^{n-1}}\Phi(\omega)\,d\nu(\omega).
$$

By assumption \eqref{eq:identity_assumption} we obtain
\begin{equation}\label{eq:nu_identity}
\int_{S^{n-1}}\omega\otimes\omega\,d\nu(\omega)=\int_S u(x)\otimes u(x)\,d\mu(x)=I.
\end{equation}

Next apply the measure disintegration theorem to decompose $\mu$ along the fibres of $u$:
there exists a family of probability (or finite) measures ${\mu_\omega}_{\omega\in S^{n-1}}$ supported on the fibres
$u^{-1}({\omega})$ such that for every nonnegative measurable $g:S\to\mathbb R$

$$
\int_S g(x)\,d\mu(x)=\int_{S^{n-1}}\Big(\int_{u^{-1}(\{\omega\})} g(x)\,d\mu_\omega(x)\Big)\,d\nu(\omega).
$$

(We may take the measures $\mu_\omega$ finite; one can normalize if desired — the argument below is unaffected.)

Define for $\nu$‑almost every $\omega\in S^{n-1}$ the function
\begin{equation}\label{eq:F_def}
F_\omega(t):=\exp\Big(\int_{u^{-1}({\omega})}\log f_x(t),d\mu_\omega(x)\Big),\qquad t\in\mathbb R.
\end{equation}
By the integrability assumptions, $F_\omega$ is well defined (possibly taking the value $0$) for $\nu$‑a.e. $\omega$ and every $t$.

For each fixed $y\in\mathbb R^n$ apply the disintegration identity to the function $x\mapsto\log f_x(\langle y,u(x)\rangle)$:
\begin{align*}
\int_S\log\bigl(f_x(\langle y,u(x)\rangle)\bigr),d\mu(x)
&=\int_{S^{n-1}}\Big(\int_{u^{-1}({\omega})}\log f_x(\langle y,\omega\rangle),d\mu_\omega(x)\Big),d\nu(\omega)\\
&=\int_{S^{n-1}}\log\bigl(F_\omega(\langle y,\omega\rangle)\bigr),d\nu(\omega).
\end{align*}
Hence the left-hand side of \eqref{eq:conclusion_S} can be written as

$$
\int_{\mathbb R^n}\exp\Big(\int_S\log f_x(\langle y,u(x)\rangle)\,d\mu(x)\Big)\,dy
=\int_{\mathbb R^n}\exp\Big(\int_{S^{n-1}}\log F_\omega(\langle y,\omega\rangle)\,d\nu(\omega)\Big)\,dy.
$$

Now apply the standard continuous Brascamp--Lieb inequality on the sphere
for the measure $\nu$ (which satisfies the normalization \eqref{eq:nu_identity}) and the family ${F_\omega}$:
\begin{equation}\label{eq:apply_BL_on_sphere}
\int_{\mathbb R^n}\exp\Big(\int_{S^{n-1}}\log F_\omega(\langle y,\omega\rangle),d\nu(\omega)\Big),dy
\le
\exp\Big(\int_{S^{n-1}}\log\Big(\int_{\mathbb R}F_\omega(t),dt\Big),d\nu(\omega)\Big).
\end{equation}

It remains to relate $\int_{\mathbb R}F_\omega$ to the integrals of the original $f_x$.
By definition of $F_\omega$ and Jensen's inequality applied to the probability measure
proportional to $\mu_\omega$ on the fibre
we have for $\nu$‑a.e. $\omega$:

$$
\int_{\mathbb R}F_\omega(t)\,dt
=\int_{\mathbb R}\exp\Big(\int_{u^{-1}(\{\omega\})}\log f_x(t)\,d\mu_\omega(x)\Big)\,dt
\le\exp\Big(\int_{u^{-1}(\{\omega\})}\log\Big(\int_{\mathbb R} f_x(t)\,dt\Big)\,d\mu_\omega(x)\Big).
$$

Substituting this bound into \eqref{eq:apply_BL_on_sphere} yields

$$
\int_{\mathbb R^n}\exp\Big(\int_{S^{n-1}}\log F_\omega(\langle y,\omega\rangle)\,d\nu(\omega)\Big)\,dy
\le
\exp\Big(\int_{S^{n-1}}\int_{u^{-1}(\{\omega\})}\log\Big(\int_{\mathbb R} f_x(t)\,dt\Big)\,d\mu_\omega(x)\,d\nu(\omega)\Big).
$$

Using the disintegration identity once again to recombine the fibre integrals into an integral over $S$ gives

$$
\int_{\mathbb R^n}\exp\Big(\int_S\log f_x(\langle y,u(x)\rangle)\,d\mu(x)\Big)\,dy
\le
\exp\Big(\int_S\log\Big(\int_{\mathbb R} f_x(t)\,dt\Big)\,d\mu(x)\Big),
$$

which is precisely \eqref{eq:conclusion_S} and completes the proof.
\end{proof}
\begin{lemma}[Lewis position]\label{lewis_lemma}
Let $(S,\sigma)$ be a finite measure space and $v:S\to\mathbb R^d$ measurable with 
\[
\int_S\|v(x)\|_2^p\,d\sigma(x)<\infty
\qquad(1< p<\infty).
\]
Define for $A\in L(\mathbb R^d)$ the functional
\[
a(A):=\Big(\int_S\|A v(x)\|_2^p\,d\sigma(x)\Big)^{1/p}
\]
Then there exists an invertible matrix $u\in GL(d)$ such that
\[
a(u)=1\qquad\text{and}\qquad a^*(u^{-1})=d,
\]
where $a^*$ is the dual norm on $L(\mathbb R^d)$ defined by $a^*(B)=\sup_{a(A)\le1}|\operatorname{tr}(B^T A)|$.
Moreover, for such $u$ the following isotropic identity holds
\begin{equation}
\int_S \|u v(x)\|_2^{p-2}\,(u v(x))\,(u v(x))^T \,d\sigma(x) \;=\; \frac{1}{d}\,I_d.
\end{equation}
\end{lemma}

\begin{proof}
Since $a$ is a norm, the existence of \(u\) with \(a(u)=1\) and dual normalization is exactly Lewis's theorem. 

For the other part, define the linear functional
\[
L(B):=\operatorname{tr}(u^{-T} B).
\]
 By the choice $a^*(u^{-1})=d$ we have
\[
\sup_{a(A)\le1} L(A)=a^*(u^{-1})=d,
\qquad\text{and}\qquad L(u)=\operatorname{tr}(u^{-T}u)=\operatorname{tr}(I)=d.
\]
Hence the linear functional $L$ attains its supremum on the unit ball $\{A:a(A)\le1\}$ at $A=u$. This implies that the matrix $\frac{1}{d}u^{-T}$ belongs to the subdifferential of $a$ at $u$:
\[
\frac{1}{d}u^{-T}\in\partial a(u).
\]
If $\Phi:=a^p$, for $p>1$ the map $a$ is Fréchet differentiable at $u$ and the subdifferential relation yields
\[
\frac{1}{d}u^{-T} \in \partial a(u) \quad\Longrightarrow\quad
\frac{p}{d} u^{-T} \in \partial \Phi(u).
\]
For $p>1$ the subdifferential of $\Phi$ is a singleton equal to the gradient, so the above becomes an equality of matrices:
\begin{equation}\label{nabla_eq}
\nabla\Phi(u) \;=\; \frac{p}{d}\,u^{-T}.
\end{equation}
The derivative of $\Phi$ at $A$ in direction $H$ is
\[
D\Phi(A)[H]
= p\int_S \|A v(x)\|_2^{p-2}\,\langle A v(x),\, H v(x)\rangle \,d\sigma(x).
\]
Writing the inner product as a trace one gets the matrix representation
of the gradient 
\[
\nabla\Phi(A) \;=\; p\int_S \|A v(x)\|_2^{p-2}\,(A v(x))\,v(x)^T\,d\sigma(x).
\]
Plugging $A=u$ into the gradient formula and equating with \eqref{nabla_eq} gives
\[
p\int_S \|u v(x)\|_2^{p-2}\,(u v(x))\,v(x)^T\,d\sigma(x) \;=\; \frac{p}{d}\,u^{-T}.
\]
Cancel the factor $p$ and transpose both sides:
\[
\int_S \|u v(x)\|_2^{p-2}\,v(x)\,(u v(x))^T\,d\sigma(x) \;=\; \frac{1}{d}\,u^{-1}.
\]
Multiply on the left by $u$ and on the right by $u^T$ to obtain the symmetric matrix identity
\[
\int_S \|u v(x)\|_2^{p-2}\,(u v(x))(u v(x))^T\,d\sigma(x) \;=\; \frac{1}{d}\,I_d,
\]
which is precisely the claimed isotropic identity.
\end{proof}

Even for the case where $p=1$, the preceding lemma remains valid. However, the requirement of an additional non-vanishing constraint, coupled with the resultant divergence in the method of proof, necessitates its independent enunciation and demonstration.

\begin{proposition}\label{Lewis p=1}
Let $(S,\sigma)$ be a finite measure space and let $v:S\to\mathbb R^d$ satisfy
\[
v\in L^1(S;\mathbb R^d),\qquad \sigma(\{x:\,v(x)=0\})=0,
\qquad \text{ess.\ span}\{v(x):x\in S\}=\mathbb R^d.
\]
Define
\[
a(A):=\int_S\|A v(x)\|_2\,d\sigma(x)\qquad(A\in L(\mathbb R^d)).
\]
Then there exists an invertible matrix $u\in GL(d)$ such that
\[
a(u)=1\qquad\text{and}\qquad a^*(u^{-1})=d.
\]
Moreover, if
\[
\zeta(x):=\frac{u v(x)}{\|u v(x)\|_2}\qquad\text{(defined a.e.)}
\]
is the canonical measurable selection, then it satisfies the matrix identity
\[
\frac{1}{d}u^{-T}=\int_S \zeta(x)\,v(x)^T\,d\sigma(x),
\]
and consequently
\[
\int_S \frac{(u v(x))\otimes (u v(x))}{\|u v(x)\|_2}\,d\sigma(x)=\frac{1}{d}I_d.
\]
\end{proposition}

\begin{proof}
As in the above proof, from Lewis theorem we obtain an invertible matrix $u\in GL(d)$ with $a(u)=1$ and
\[
\frac{1}{d}u^{-T}\in\partial a(u).
\]
Set \(\rho:\mathbb R^d\to\mathbb R,\ \rho(t)=\|t\|_2\). The subdifferential \(\partial\rho(t)\) is defined by
\[
\partial\rho(t)=\{s\in\mathbb R^d:\ \rho(t+h)-\rho(t)\ge\langle s,h\rangle\ \text{for all }h\in\mathbb R^d\}.
\]
We verify the two standard facts:

(i) If \(t\neq0\) then \(\rho\) is differentiable at \(t\) and
\[
\partial\rho(t)=\Big\{\frac{t}{\|t\|_2}\Big\}.
\]
Indeed, the directional derivative at \(t\) in direction \(h\) equals \(\langle t/\|t\|_2,h\rangle\), so the only possible subgradient is \(t/\|t\|_2\), and it satisfies the defining inequality.

(ii) If \(t=0\) then
\[
\partial\rho(0)=\{s\in\mathbb R^d:\ \|s\|_2\le1\}.
\]
Indeed, if \(s\in\partial\rho(0)\) then \(\|h\|_2\ge\langle s,h\rangle\) for all \(h\). Taking \(h\) on the unit sphere gives \(\langle s,u\rangle\le1\) for all \(\|u\|_2=1\), hence \(\|s\|_2\le1\). Conversely, if \(\|s\|_2\le1\) then \(\langle s,h\rangle\le\|s\|_2\,\|h\|_2\le\|h\|_2\) by Cauchy–Schwarz, so \(s\in\partial\rho(0)\).
\smallskip

\noindent Since \(u\) is invertible and \(\sigma(\{v=0\})=0\), we have \(u v(x)\neq0\) for a.e.\ \(x\). Hence for a.e.\ \(x\) the set \(\partial\rho(u v(x))\) is the singleton \(\{u v(x)/\|u v(x)\|_2\}\). Define
\[
\zeta(x):=\frac{u v(x)}{\|u v(x)\|_2}\quad\text{for a.e.\ }x.
\]
This \(\zeta\) is measurable because \(x\mapsto u v(x)\) is measurable and \(t\mapsto t/\|t\|_2\) is continuous on \(\mathbb R^d\setminus\{0\}\) and \(\|\zeta(x)\|_2=1\) a.e.

The hypothesis \(\tfrac{1}{d}u^{-T}\in\partial a(u)\) means that for every matrix \(H\in L(\mathbb R^d)\)
\begin{equation}\label{SG}
a(u+H)-a(u)\;\ge\;\Big\langle\frac{1}{d}u^{-T},\,H\Big\rangle_{\mathrm{HS}}.
\end{equation}
Fix an arbitrary matrix $H\in L(\mathbb R^d)$ and set
\[
\phi(t):=a(u+tH)=\int_S \|u v(x)+t\,H v(x)\|_2\,d\sigma(x),\qquad t\in\mathbb R.
\]
By convexity the one–sided derivative $\displaystyle\phi'_+(0)=\lim_{t\to 0^+}\frac{\phi(t)-\phi(0)}{t}$ exists (possibly $+\infty$).  
For a.e.\ $x$ the scalar function $g_x(t):=\|u v(x)+t\,H v(x)\|_2$ is convex in \(t\) and its right derivative at \(0\) equals
\[
g'_{x^{+}}(0)=\big\langle\zeta(x),\,H v(x)\big\rangle,
\]
because $\zeta(x)$ is the unique element of $\partial\|\cdot\|(u v(x))$ for a.e.\ \(x\). Moreover, since $|\langle\zeta,Hv\rangle|\le\|H\|_{\mathrm{op}}\|v\|_2\in L^1$ dominated convergence  allows passing the derivative inside the integral, hence
\begin{equation}\label{A}
\phi'_+(0)=\int_S \big\langle\zeta(x),\,H v(x)\big\rangle\,d\sigma(x). 
\end{equation}

On the other hand, the subgradient inclusion $\tfrac{1}{d}u^{-T}\in\partial a(u)$ implies that for every $t>0$
\[
\frac{\phi(t)-\phi(0)}{t} \ge \Big\langle\frac{1}{d}u^{-T},\,H\Big\rangle_{\mathrm{HS}}.
\]
Passing to the limit $t\to 0^{+}$ we obtain
\begin{equation}\label{B}
\phi'_+(0)\ge \Big\langle\frac{1}{d}u^{-T},\,H\Big\rangle_{\mathrm{HS}}.    
\end{equation}
Applying the same reasoning to $-H$ gives
\begin{equation}\label{C}
\phi'_+(0)\le \Big\langle\frac{1}{d}u^{-T},\,H\Big\rangle_{\mathrm{HS}}.
\end{equation}
Combining \eqref{B} and \eqref{C} yields equality
\begin{equation}\label{D}
\phi'_+(0) = \Big\langle\frac{1}{d}u^{-T},\,H\Big\rangle_{\mathrm{HS}}.
\end{equation}
Comparing \eqref{A} and \eqref{D} we obtain, for every $H$,
\[
\Big\langle\frac{1}{d}u^{-T},\,H\Big\rangle_{\mathrm{HS}}
= \int_S \big\langle\zeta(x),\,H v(x)\big\rangle\,d\sigma(x).
\]
Since both sides are continuous linear functionals of $H$, the above identity implies the matrix equality
\[
\frac{1}{d}u^{-T} = \int_S \zeta(x)\,v(x)^T\,d\sigma(x).
\]
If we take the transpose of the matrix identity and conjugate by \(u\):
\[
\frac{1}{d}u^{-1} = \int_S v(x)\,\zeta(x)^T\,d\sigma(x).
\]
Multiplying on the left by \(u\) and on the right by \(u^T\) gives
\[
\frac{1}{d}I_d
= u\Big(\int_S v(x)\,\zeta(x)^T\,d\sigma(x)\Big)u^T
= \int_S (u v(x))\,\zeta(x)^T\,u^T.
\]
Substituting \(\zeta(x)=\dfrac{u v(x)}{\|u v(x)\|_2}\) yields
\[
\frac{1}{d}I_d
= \int_S (u v(x))\,\frac{(u v(x))^T}{\|u v(x)\|_2}\,d\sigma(x)
= \int_S \frac{(u v(x))(u v(x))^T}{\|u v(x)\|_2}\,d\sigma(x),
\]
which is the desired isotropic identity.
This completes the proof.
\end{proof}

\bigskip\noindent\textit{Remark.}
\begin{itemize}
\item Passing the pointwise inequality to the integral is justified because \(|\langle \zeta(x),H v(x)\rangle|\le\|H\|_{\mathrm{op}}\|v(x)\|_2\) and \(v\in L^1\).
\item The uniqueness and measurability of \(\zeta(x)\) follow from \(u v(x)\neq0\) a.e.; if \(u v=0\) on a positive–measure set one must interpret the result as an assertion about some measurable selection from the unit ball there (the same proof works but requires a measurable–selection remark).  
\end{itemize}
\end{document}